\documentclass[a4paper,11pt,oneside  ]{article}	

\makeindex
\usepackage{hyperref}
\usepackage{graphicx}
\usepackage{enumerate}
\usepackage{todonotes}
\hypersetup{colorlinks, linkcolor=blue, urlcolor=red, filecolor=green, citecolor=blue}
\usepackage{endnotes}
\usepackage{amsmath}
\usepackage{amsthm}
\usepackage{amssymb}
\usepackage{authblk}
\usepackage{enumerate}
\usepackage{fullpage}
\usepackage{esint}
\usepackage{subfigure}

\theoremstyle{plain}
\newtheorem{theorem}{Theorem}[section]

\newtheorem{lemma}[theorem]{Lemma}
\newtheorem{proposition}[theorem]{Proposition}

\theoremstyle{definition}
\newtheorem{definition}[theorem]{Definition}

\newtheorem{assumption}[theorem]{Assumption}

\theoremstyle{remark}
\newtheorem{remark}{Remark}

\newcommand{\N}{\mathbb{N}}
\newcommand{\R}{\mathbb{R}}

\newcommand\e{\varepsilon}

\def\Xint#1{\mathchoice
   {\XXint\displaystyle\textstyle{#1}}%
   {\XXint\textstyle\scriptstyle{#1}}%
   {\XXint\scriptstyle\scriptscriptstyle{#1}}%
   {\XXint\scriptscriptstyle\scriptscriptstyle{#1}}%
   \!\int}
\def\XXint#1#2#3{{\setbox0=\hbox{$#1{#2#3}{\int}$}
     \vcenter{\hbox{$#2#3$}}\kern-.5\wd0}}

\def\fint{\Xint-}

\def\H{\boldsymbol H}
\newcommand\esssup{\mathop{\operatorname{ess\,sup}}}
\newcommand\essinf{\mathop{\operatorname{ess\,inf}}}

\newcommand\dist{\operatorname{dist}}
\newcommand\sym{\operatorname{sym}}

\newcommand{\proofref}[1]{{\it(For the proof see Section~\ref{#1}).}}

\newcommand{\SO}[1]{\operatorname{SO}(#1)}
\newcommand{\Skew}[1]{\operatorname{Skew}(#1)}
\newcommand{\Sym}[1]{\operatorname{Sym}(#1)}

\newcommand\Id{\operatorname{Id}}
\newcommand{\wto}{\rightharpoonup}

\newcommand\ho{{\operatorname{hom}}}

\newcommand{\step}[1]{\medskip\noindent\textbf{Step #1. }}
\newcommand{\substep}[1]{\medskip\noindent\textit{Substep #1. }}

\newcommand{\Eop}{{\boldsymbol{\rm E}}}
\newcommand{\bfx}{{\bar{\boldsymbol x}}}
 \title{Derivation of a homogenized bending--torsion theory for rods with micro-heterogeneous prestrain}
\date{\today}

\author[1]{Robert Bauer\thanks{robert.bauer3@tu-dresden.de}}
\author[1]{Stefan Neukamm\thanks{stefan.neukamm@tu-dresden.de}}
\author[2]{Mathias Sch\"affner\thanks{mathias.schaeffner@math.uni-leipzig.de}}
\affil[1]{Faculty of Mathematics, Technische Universit\"at Dresden}
\affil[2]{Mathematisches Institut, Universit{\"a}t Leipzig}

\begin{document}

\maketitle

\begin{abstract}
  In this paper we investigate rods made of nonlinearly elastic, composite--materials that feature a micro-heterogeneous prestrain that oscillates (locally periodic) on a scale that is small compared to the length of the rod. As a main result we derive a homogenized bending--torsion theory for rods as $\Gamma$-limit from 3D nonlinear elasticity by simultaneous homogenization and dimension reduction under the assumption that the prestrain is of the order of the diameter of the rod. The limit model features a spontaneous curvature--torsion tensor that captures the macroscopic effect of the micro-heterogeneous prestrain. We devise a formula that allows to compute the spontaneous curvature--torsion tensor by means of a weighted average of the given prestrain. The weight in the average depends on the geometry of the composite and invokes correctors that are defined with help of  boundary value problems for the system of linear elasticity. The definition of the correctors depends on a relative scaling parameter $\gamma$, which monitors the ratio between the diameter of the rod and the period of the composite's microstructure. We observe an interesting size-effect: For the same prestrain a transition from flat minimizers to curved minimizers occurs by just changing the value of $\gamma$. Moreover, in the paper we analytically investigate the microstructure-properties relation in the case of isotropic, layered composites, and consider applications to nematic liquid--crystal--elastomer rods and shape programming.
\smallskip

\end{abstract}
\tableofcontents

\section{Introduction}

\paragraph{Motivation.} Residual stress can have a tremendous effect on the mechanical behavior of slender elastic structures: Equilibrium states of elastic thin films and rods with residual stresses often have a complex shape in equilibrium, and may feature wrinkling and symmetry breaking. Many natural and synthetic materials feature residual stresses due to different physical principles, e.g., growth of soft tissues \cite{amar2005growth, dervaux2008morphogenesis}, swelling and de-swelling in polymer gels \cite{ionov2013biomimetic}, thermo-mechanical coupling in nematic liquid crystal elastomers \cite{Warner}, and thermal expansion in production processes. These mechanisms may be triggered by different stimuli (such as temperature, light, and humidity), and are exploited in the design of active thin structures---elastic structures that are capable to change from an initially flat state into a 3D ``programmed'' configuration  in response to external stimuli, see \cite{B_Efrati07} and \cite{van2018programming} for a recent review on \textit{shape shifting} flat soft matter. Modeling of such structures, requires (next to a description of the stimuli process)
a good understanding of the highly nonlinear relation between residual stresses and the geometry of the equilibrium shape. Although intensively studied, no satisfying understanding of this relation has been obtained so far. This is especially the case for composite materials, where material properties and residual stresses feature microstructure---a situation that is relevant for future applications, since ``\textit{Shape-changing materials offer a powerful tool for the incorporation of sophisticated planar micro- and nano-fabrication techniques in 3D constructs}'' as pointed out in \cite{van2018programming}.

\paragraph{Overview of results.} In this paper we investigate rods made of nonlinearly elastic, composite--materials that feature a micro-heterogeneous prestrain (or~residual stress) that oscillates (locally periodic) on a scale that is small compared to the length of the rod. Our starting point is the energy functional of $3D$-nonlinear elasticity with a cylindrical reference domain $\Omega_h=(0,\ell)\times hS\subset\R^3$:
\begin{equation}\label{I:1}
  u\mapsto \int_{\Omega_h}W_\e(x,\nabla u(x) A_{\e,h}^{-1}(x))\,dx.
\end{equation}
It depends on two small parameters $h$ and $\e$ (describing the thickness of the rod and the period of the composite), and describes prestrain with help of a tensor field $A_{\e,h}$, see Section~\ref{sec:2.1} for the continuum--mechanical interpretation. We suppose $W_{\e}$ to describe a non-degenerate, nonlinear  material with stress--free reference state. Moreover, we assume that the amplitude of the prestrain is comparable to the diameter of the rod, i.e., $A_{\e,h}=\Id+O(h)$ so that $A_{\e,h}^{-1}\approx \Id+B_{\e,h}$ with a tensor field $B_{\e,h}$ that is uniformly bounded in $\e$ and $h$. We suppose that both, the prestrain tensor $B_{\e,h}$ and the elasticity tensor $\mathbb L_\e$ (obtained by linearization of $W_\e$ at the identity) converge in a two-scale sense, see Section~\ref{sec:2.2} for the precise definition.

As a main result (see Theorem~\ref{Th:gamma0}) we derive the $\Gamma$-limit as $(\e,h)\downarrow 0$ of \eqref{I:1} in the bending regime. In this \textit{simultaneous homogenization and dimension reduction limit}, we obtain a \textit{homogenized} bending--torsion theory for rods that features a \textit{spontaneous curvature--torsion tensor} $K_{\rm eff}:(0,\ell)\to\Skew 3$. It captures the macroscopic effect of the micro-heterogeneous prestrain:
\begin{equation}\label{I:2}
  (u,R)\mapsto \int_0^\ell Q_{\rm hom}\big(x_1,(R^t\partial_1R)-K_{\rm eff}(x_1)\big)\,dx_1,
\end{equation}
where bending and torsion  of the rod is described by the isometry $u\in W^{2,2}_{\rm iso}((0,\ell),\R^3)$ and an attached orthonormal frame $R\in W^{1,2}((0,\ell);\SO 3)$, $Re_1=\partial_1 u$. The elastic moduli of the rod are described by the quadratic form $Q_{\rm hom}$. It is positive definite on skew symmetric matrices and can be computed by a linear relaxation and homogenization formula from $\mathbb L_\e$---the fourth order elasticity tensor obtained by linearizing $W_\e$ at identity. While it is difficult to study energy minimizers of \eqref{I:1} directly, energy minimizers of \eqref{I:2} can easily be obtained by integrating the spontaneous curvature--torsion field $K_{\rm eff}$. It turns out that $K_{\rm eff}$ depends on the two-scale limit of $\mathbb L_\e$ (nonlinearly) and on $B_{\e,h}$ (linearly). In addition, we observe that both $Q_{\rm hom}$ and $K_{\rm eff}$ depend on the relative--scaling parameter $\gamma=\lim_{(\e,h)\downarrow 0}\frac{h}{\e}$. 

Next to the $\Gamma$-convergence result, we introduce an effective scheme to evaluate $Q_{\rm hom}$ and $K_{\rm eff}$ which invokes the definition of suitable correctors that are characterized by corrector equations that essentially come in form of boundary value problems for the system of linear elasticity, see Proposition~\ref{P:BVP}. The spontaneous curvature--torsion tensor $K_{\rm eff}$ is obtained as weighted average of the prescribed prestrain tensor with weights given by the correctors. For isotropic composites with a laterally layered microstructure, we can solve the corrector equations by hand and we obtain explicit formulas for the $Q_{\rm hom}$ and $K_{\rm eff}$, see Lemma~\ref{L:isotropic}. We observe a significant qualitative and quantitative dependence of $K_{\rm eff}$ on the relative--scaling parameter $\gamma$. In particular, we device an example of a prestrain that yields a transition from a straight minimizer (i.e., $K_{\rm eff}\equiv 0$) to a curved minimizers (i.e., $K_{\rm eff}\neq 0$) by only changing the value of $\gamma$, see Section~\ref{sec:4.2}. Moreover, we briefly discuss applications to nematic liquid crystal elastomers in Section~\ref{par:liquid}, and shape programming in Section~\ref{sec:4.4}.%, and comment on potential applications to rods with varying cross-section in Remark~\ref{rem:varyingS}.

\paragraph{Survey of the literature.} The derivation of mechanical models for rods has a long history. For modeling based on equilibria of forces or conservation of momentum, and derivations via formal asymptotic expansions or based on the assumption of a kinematic ansatz we refer the reader to \cite{antman1976ordinary, audoly, ciarlet1997mathematical,Mielke}. In contrast to these works, we take the perspective of energy minimization, and our result is  an ansatz-free derivation that is based on the $\Gamma$-convergence methods developed by Friesecke, James \& M\"uller in \cite{FJM02}, in particular the geometric rigidity estimate. If we replace in \eqref{I:1}  the prestrain tensor $A_{\e,h}$ by the identity matrix, then we recover a standard 3D nonlinear elasticity model  \textit{without prestrain}, i.e., with a stress--free reference configuration. In that case the limit $h\downarrow 0$ with $\e>0$ fixed, corresponds to a dimension reduction problem (without homogenization) studied by Mora \& M\"uller in \cite{MM03} where for the first time a bending--torsion theory for inextensible rods has been derived via $\Gamma$-convergence. On the other hand, the limit $(h,\e)\downarrow 0$ corresponds to \textit{simultaneous homogenization and dimension reduction} and is studied by the second author in \cite{Neu10, N12}, see also \cite{B_NO15, NV14, HNV14, Vel15} where the same problem for plates is considered. First results that combine dimension reduction in the presence of a prestrain are due to Schmidt: In \cite{Sch07, Schm07a} prestrained bending plates are obtained from 3D nonlinear elasticity; see also \cite{LMP10} on the derivation of a model for  prestrained von K\'arm\'an plates, and \cite{ADeS16} where applications to models for nematic liquid crystal elastomers are studied.
 Our result can be viewed as a combination of Schmidt's work with \cite{Neu10, N12}. We note that a simplified version of our main result is announced in the second author's thesis \cite{Neu10} (together with a rough sketch of the proof). Recently, the derivation of prestrained bilayer rods has been investigated by Kohn \& O'Brien \cite{KOB17} and Cicalese, Ruf \& Solombrino \cite{CRS}. In these interesting works not only energy minimizers are studied, but also the convergence of critical points is established and a comparison with experiments \cite{Shtuk} is discussed. Another interesting direction of active research on related topics are the derivation and analysis of ribbons, e.g., \cite{ADeSK17, FHMP16,FHMP16a}.

In the results discussed so far the prestrain (if present) is assumed to be infinitesimally small. In the last decade, dimension reduction for finite prestrain (yet smoothly varying on a macroscopic scale) has been studied in the framework of non-Euclidean elasticity theory \cite{B_Efrati09}, e.g.,  \cite{KS14,LMP10,LP11, BLS16, Lewicka18} for the derivation of non-Euclidean theories for rods and plates. Rods and shells with nontrivially curved reference configuration lead to similar models when being pulled back to a flat reference configuration (cf.~Remark~3 below), e.g., see \cite{Sca06, Vel13,HV15,HV18} for shells. We refer to \cite{Bartels} for a recent review on numerical simulation methods for rods and plate models.

\paragraph{Structure of the paper.}
\smallskip
We introduce the general framework in Section~\ref{sec:2}. In particular, we explain the modeling of prestrained composites (based on a multiplicative decomposition of the strain) in Section~\ref{sec:2.1}. The 3D model and its limit are described in Sections~\ref{sec:2.2} and \ref{sec:2.3}. In Section~\ref{sec:2.4} we present an abstract definition of the homogenization and averaging formulas that determine $Q_{\rm hom}$ and $K_{\rm eff}$. In Proposition~\ref{P:BVP} in Section~\ref{S:algo} we describe the effective evaluation scheme for these formulas. It is based on the notion of suitable correctors. Eventually, in Section~\ref{s:isotropic} we discuss various applications of the theory to isotropic material for which the correctors, homogenization-, and averaging formulas can be evaluated by hand. All proofs are contained in Section~\ref{S:proofs}.

\subsection{Notation}
\begin{itemize}
\item $e_1,e_2,e_3$ denotes the standard basis of $\R^3$.
\item Given $a,b\in\R^d$ we write $a\otimes b$ to denote the unique matrix in $\R^{d\times d}$ given by $(a\otimes b)c=(b\cdot c)a$ for all $c\in\R^d$.
\item We write $\Sym d$, $\Skew d$, and $\SO d$ for the space of symmetric, skew-symmetric, and rotation matrices in $\R^{d\times d}$. We denote the identity matrix by $\Id$.
\item We decompose $x=(x_1,\bar x)\in\R^3$ into the \textit{in-plane component} $x_1:=x\cdot e_1$ and the \textit{out-of-plane components} $\bar x:=(x_2,x_3):=(x\cdot e_2,x\cdot e_3)$. 
\item For all $x\in\R^3$ we set $\bfx(x):=\sum_{i=2,3}x_ie_i\in\R^3$. We tacitly drop the argument and simply write $\bfx$ (instead of $\bfx(x)$).
\end{itemize}

\section{General framework and statement of main results}\label{sec:2}
In this section we state the general framework and our main result. 

\subsection{A model for prestrain in nonlinear elasticity.} \label{sec:2.1}

We start by presenting a model for prestrained composites in nonlinear elasticity. We first introduce a class of stored energy functions:
\begin{definition}[Nonlinear and linearized material law]\label{D1}
  Let $0<\alpha\leq\beta$, $\rho>0$, and let $r:[0,\infty)\to[0,\infty]$ denote a monotone function satisfying $\lim_{\delta\to0}r(\delta)=0$. 
\begin{itemize}
\item The class $\mathcal W(\alpha,\beta,\rho,r)$ consists of all measurable functions $W:\R^{3\times 3}\to[0,+\infty]$ such that,
\begin{itemize}
 \item[(W1)] $W$ is frame indifferent: $W(RF)=W(F)$ for all $F\in\R^{3\times 3}$, $R\in \SO 3$.
 \item[(W2)] $W$ is non degenerate:
 \begin{eqnarray*}
  W(F)&\geq& \alpha \dist^2(F,\SO 3)\qquad\mbox{for all $F\in\R^{3\times 3}$,}\\
  W(F)&\leq& \beta \dist^2(F,\SO 3)\qquad\mbox{for all $F\in\R^{3\times 3}$ with $\dist^2(F,\SO 3)\leq \rho$.}
 \end{eqnarray*}
 \item[(W3)] $W$ is minimal at $\Id$: $W(\Id)=0$.
 \item[(W4)]$W$ admits a quadratic expansion at $\Id$:
  \begin{equation*}
   |W(\Id+G)-Q(G)|\leq |G|^2r(|G|)\qquad\mbox{for all $G\in\R^{3\times 3}$,}
  \end{equation*}
  where $Q:\R^{3\times 3}\to\R$ is a quadratic form.
\end{itemize}
\item The class $\mathcal Q(\alpha,\beta)$ consists of all quadratic forms $Q$ on $\R^{3\times 3}$ such that
  \begin{equation*}
    \forall G\in\R^{3\times 3}\,:\qquad\alpha|\sym G|^2\leq Q(G)\leq\beta|\sym G|^2.
  \end{equation*}
  We associate with $Q$ the fourth order tensor $\mathbb L\in\mbox{Lin}(\R^{3\times 3},\R^{3\times 3})$ defined by the polarization identity $\langle \mathbb L F,G\rangle:=\frac12\big(Q(F+G)-Q(F)-Q(G)\big)$.
\end{itemize}
\end{definition}
Stored energy functions of class $\mathcal W(\alpha,\beta,\rho,r)$ describe materials that have a \textit{stress-free} reference state (cf.~$(W3)$), and that can be linearized at that state  (e.g., in the sense of $\Gamma$-convergence, see \cite{DNP, MN11, GN11, Neu10}). The elastic moduli of the linearized model are given by the quadratic form $Q$ in condition $(W4)$, and we have:
\begin{lemma}[see Lemma~2.7 in \cite{N12}]
 Let $W\in\mathcal W(\alpha,\beta,\rho,r)$ and denote by $Q$ the quadratic form in $(W4)$. Then $Q\in\mathcal Q(\alpha,\beta)$.
\end{lemma}
We describe prestrained  composites with help of a multiplicative decomposition of the strain. To motivate this decomposition, we consider for a moment a composite consisting of two materials. We suppose that each of the materials can be described w.r.t.~their \textit{individual} stress-free reference configurations by stored energy functions
 $W_1,W_2\in\mathcal
 W(\alpha,\beta,\rho,r)$, respectively. Let
 $\Omega=\Omega_1\dot\cup\Omega_2\subset\R^3$ denote a \textit{common} reference 
 configuration of the composite and suppose that material--one
 (resp.~--two) occupies the subdomain $\Omega_1$ (resp.~$\Omega_2$). We
 suppose that material--one is stress-free in the reference
 configuration $\Omega_1$, and thus the elastic energy coming from material--one is captured by $\int_{\Omega_1}W_1(\nabla u)$. On the other hand, we suppose that material--two is prestrained in the following sense: If
 we separate an (infinitesimally small) test-volume $U\subset \Omega_2$ from the rest
 of the body, then it relaxes to a stress-free (energy minimizing) state described by an affine deformation $x\mapsto \widetilde Ax$ where $\widetilde A\in\R^{3\times 3}$ is positive definite and independent of $U$, see Figure~\ref{prestrain} for illustration. Thus, $\widetilde\Omega:=\widetilde A\Omega_2$ defines an alternative,  stress-free reference state for material--two, and the elastic energy of a deformation $\tilde u$ defined relative to $\widetilde\Omega$ is given by $\int_{\widetilde\Omega}W_2(\nabla\tilde u)\,d\tilde x$. Since the original deformation $u:\Omega\to\R^3$ and $\tilde u$ are related by $u(x):=\tilde u(\widetilde Ax)$ (for $x\in\Omega_2$), we deduce that the energy functional on the level of $u$ associated with material--two is given by
  \begin{equation*}
    \int_{\widetilde\Omega}W_2(\nabla\tilde u)=    \int_{\widetilde A\Omega_2}W_2\Big((\nabla u(\widetilde A^{-1}\tilde x))\widetilde A^{-1}\Big)\,d\tilde x=\int_{\Omega_2}W_2((\nabla u) \widetilde A^{-1})\det \widetilde A\,dx.
  \end{equation*}
  Hence, the energy functional for the whole composite takes the form
  \begin{align*}
    &\mathcal E(u):=\int_\Omega W(x,\nabla u(x) A^{-1}(x))\,dx,\\
    &\qquad W(x,F):=
    \begin{cases}
      W_1(F)\det (A(x))&x\in\Omega_1,\\
      W_2(F)\det (A(x))&x\in\Omega_2.
    \end{cases},\qquad A(x):=    \begin{cases}
      \Id&x\in\Omega_1,\\
      \widetilde A&x\in\Omega_2.
    \end{cases}
  \end{align*}

\begin{figure}
\centering
\includegraphics[scale=0.5]{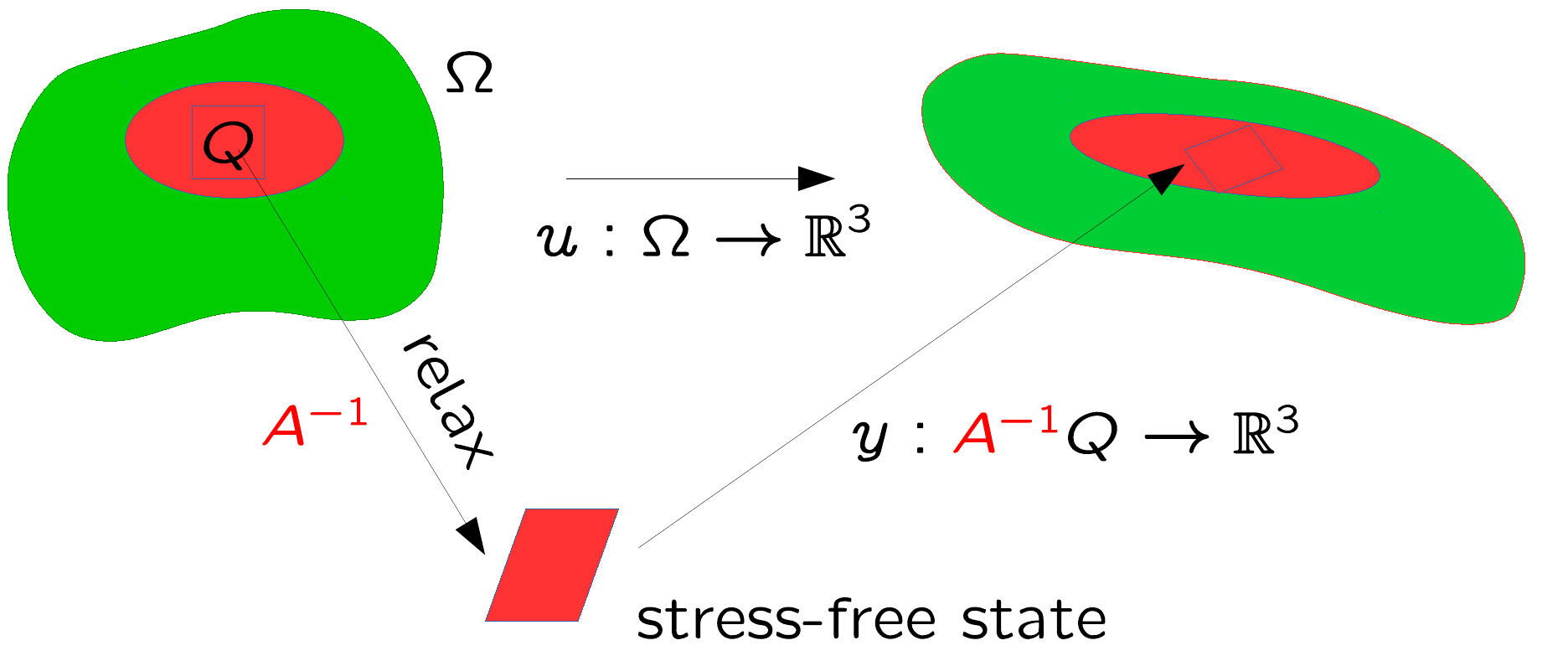}
  \caption{\small{Schematic picture of the multiplicative decomposition of the strain}}\label{prestrain}
\end{figure}

 This corresponds to a \textit{multiplicative decomposition} $F=F^{\rm el}A^{-1}$ of the strain. Similar decompositions are used in models for finite strain elasto-plasticity \cite{L69} (where $A$ is called the plastic strain tensor and is given by a flow rule), or in biomechanical models for growth and remodeling of tissues and plants, e.g., see \cite{RHM94, GMV10}.

  If the prestrain is small, then we can simplify the decomposition: Suppose that $A=R(\Id-h B)$ with $R\in\SO 3$, $B\in\R^{3\times 3}$, and $h>0$. Then for $h\ll 1$, $A$ can be inverted by the Neumann Series $A^{-1}=\Big(\sum_{k=0}^\infty(hB)^k\Big)R^{-1}=(\Id+hB)R^{-1}+O(h^2)$. Moreover, $\det(A)=\det(\Id-hB)=1+O(h)$. Hence, we arrive at an energy functional of the form $\int_{\Omega}W\Big(x,\nabla u(x)(\Id+hB(x))\Big)\,dx$ with $W(x,\cdot)\in\mathcal W(\alpha,\beta,\rho,r)$ and a tensor $B(x)\in\R^{3\times 3}$. The functional describes (up to an error of order smaller than $h^2$) a composite material with heterogeneous prestrain $(\Id-hB(\cdot))$.

\subsection{The three-dimensional model.} \label{sec:2.2}
Let $S\subset \R^2$ be a Lipschitz domain (open, bounded and connected)---the cross-section of the rod. We may assume without loss of generality that
\begin{equation}\label{ass:S}
  \int_S x_2=\int_Sx_3=\int_Sx_2x_3=0,
\end{equation}
(this can always be achieved by applying a rigid motion). Set $\omega:=(0,\ell)$. We denote by $\Omega_h:=\omega\times hS$ the reference configuration of the rod with thickness $\sim h>0$. For our purpose it is convenient to describe the deformation w.r.t.~the rescaled reference domain $\Omega:=\omega\times S$, and thus consider for $u:\Omega\to\R^3$ the \textit{scaled deformation gradient},
\begin{equation*}
  \nabla_hu(x)=\big(\partial_1u(x),\tfrac{1}{h} \bar \nabla u(x)\big),\qquad \bar\nabla u(x):=(\partial_2 u(x),\partial_3u(x)).
\end{equation*}
Rescaling \eqref{I:1} and assuming that the prestrain takes the form $A_{\e,h}=(\Id+hB_{\e,h})^{-1}$ yields an energy functional of the form $\mathcal I^{\e,h}:L^2(\Omega)\to[0,+\infty]$,
\begin{equation}\label{def:ene}
  \mathcal I^{\e,h}(u):=
  \begin{cases}
    \frac1{h^2}\int_{\Omega} W_\e(x,\nabla_h u(x)(\Id+hB_{\e,h}(x)))\,dx&\text{if }u\in H^1(\Omega),\\
    +\infty&\text{else.}
  \end{cases}
\end{equation}
This parametrized energy functional is the starting point of our derivation. We make the following assumption on the material law:
\begin{assumption}[Material law] \label{ass:W}
  Let $\alpha,\beta,\rho,r$ be fixed (as in Definition~\ref{D1}). Let $W_\e:\Omega\to[0,+\infty]$ be a sequence of Borel-functions such that,
  \begin{enumerate}[(i)]
  \item $W_\e(x,\cdot)\in\mathcal W(\alpha,\beta,\rho,r)$ for almost every $x\in\Omega$ and for every $\e>0$.
  \end{enumerate}
  We suppose that there exists $Q:\Omega\times\R\times\R^{3\times 3}\to\R$ such that 
  \begin{enumerate}
  \item[(ii)] $Q(x_1,\bar x,y,\cdot)$ is a quadratic form that is piecewise continuous in $x_1$ and periodic in $y$. More precisely,
    \begin{enumerate}
    \item $Q(x,y,\cdot)\in\mathcal Q(\alpha,\beta)$ for a.e.~$x\in\Omega$, $y\in\R$,
    \item $Q(\cdot,G)$ is $\mathcal B(\omega)\otimes\mathcal L(S\times\R)$-measurable for all $G\in\R^{3\times 3}$,
    \item The fourth order tensor $\mathbb L=\mathbb L(x_1,\bar x, y)$ associated with $Q$ (cf. Definition~\ref{D1}) satisfies
      \begin{equation*}
        \omega\ni x_1\mapsto \mathbb L(x_1,\cdot)\in L^\infty(S\times\R;\mbox{Lin}(\R^{3\times 3};\R^{3\times 3}))\text{ is piecewise continuous.}
      \end{equation*}
    \item $y\mapsto Q(x,y,G)$ is periodic for a.e.~$x\in\Omega$ and $G\in\R^{3\times 3}$ .
    \end{enumerate}
  \item[(iii)] The quadratic expansion  at identity $Q_\e(x,\cdot)$ of $W_\e(x,\cdot)$ (cf.~$(W4)$) satisfies
    \begin{equation*}
      \limsup_{\e\to0}\esssup_{x\in\Omega}\max_{G\in\R^{3\times 3}\atop |G|=1}|Q_\e(x,G)-Q(x,\tfrac{x_1}\e,G)|=0.
    \end{equation*}
 \end{enumerate}
\end{assumption}
Regarding the prestrain, we suppose that $B_{\e,h}$ is locally periodic. Our precise assumption on $B_{\e,h}$ involves the notion of two-scale convergence in a variant for slender domains \cite{Neu10,N12} (see \cite{Nguetseng,Allaire} for the original definition of two-scale convergence). Since this variant of two-scale convergence is sensitive to the relative scaling between $h$ and $\e$, we introduce a parameter $\gamma\in[0,\infty]$ describing the \textit{relative scaling} of $h$ and $\e$.
\begin{assumption}[Relative scaling of $h$ and $\e$]\label{A:gamma}
  We suppose that there exists $\gamma\in[0,\infty]$ and a monotone function $\e:(0,\infty)\to (0,\infty)$ such that $\lim_{h\downarrow 0}\e(h)=0$ and $\lim_{h\downarrow 0}\frac{h}{\e(h)}=\gamma$. 
\end{assumption}
\begin{definition}[Two-scale convergence]%\label{def:twoscaleconv}
  Let $Y:=[0,1)$ and denote by $\mathcal Y:=\R/Y$ the one-dimensional torus.
  We say a sequence $(g^h)\subset L^p(\Omega)$, $p\in[1,\infty)$, weakly two-scale converges in $L^p$ to a function $g\in L^p(\Omega\times \mathcal Y)$ as $h\to0$, if $(g^h)$ is bounded in $L^p(\Omega)$ and 
  \begin{equation*}
    \forall\psi\in C_c^\infty(\Omega;C(\mathcal Y))\,:\qquad\limsup_{h\to0}\int_\Omega g^h(x)\psi(x,\tfrac{x_1}{\e(h)})\,dx=\iint_{\Omega\times Y}g(x,y)\psi(x,y)\,dy\,dx,
  \end{equation*}
  where $h\mapsto\e(h)$ is as in Assumption~\ref{A:gamma}. We say $(g^h)$ strongly two-scale converges to $g$ if additionally $\|g^h\|_{L^p(\Omega)}\to\|g\|_{L^p(\Omega\times Y)}$.
  We write $g^h\stackrel{2}{\wto} g$ in $L^p$ (resp.\ $g^h\stackrel{2}{\longrightarrow} g$) for weak (resp.\ strong) two-scale convergence in $L^p$.
\end{definition}
\begin{remark}
  Note that this notion of two-scale convergence changes if we change
  the parameter $\gamma$. A prototypical example of a strongly two-scale
  convergent sequence is as follows: Let $g\in L^2(\Omega;C(\mathcal
  Y))$, then $g^h(x):=g(x,\frac{x_1}{\e(h)})$ strongly two-scale
  converges in $L^2$ to $g$.
\end{remark}
\begin{assumption}[Prestrain]\label{ass:B}
  We suppose that there exists $B\in L^2(\Omega\times\mathcal Y,\R^{3\times 3})$ such that
  \begin{equation}\label{ass:eqB}
    \begin{split}
      \limsup_{h\to0}h\|B_{\e(h),h}\|_{L^\infty(\Omega)}=0\qquad\text{and}\qquad B_{\e(h),h}\stackrel{2}{\rightarrow} B\text{ in $L^2$}.
    \end{split}
  \end{equation}
\end{assumption}

\subsection{Limiting model and $\Gamma$-convergence.} \label{sec:2.3}

Under the assumptions above,  we can pass to the $\Gamma$-limit of  $\mathcal I^{\e,h}$ as $(\e,h)=(\e(h),h)\to 0$. We  obtain as a limit a functional defined on the
the set $\mathcal A$ of all deformations of the rod that describe (length-preserving) bending- and twisting-deformations, and an infinitesimal stretch:
\begin{align}\label{def:A}
  \mathcal A:=\Big\{\,(u,R,a)\,:\,&u\in W^{2,2}(\omega;\R^3),\,R\in W^{1,2}(\omega;\R^{3\times 3})\cap L^2(\omega;\SO 3),\,\partial_1 u=Re_1,\\
            &a\in L^2(\omega)\,\Big\}.\notag
\end{align}
 The $\Gamma$-limit is given by $\mathcal I:\mathcal A\to[0,\infty)$,
\begin{equation}\label{def:limitI}
  \mathcal I(u,R,a):=\int_\omega  Q_{\rm hom}(x_1,R^t(x_1)\partial_1R(x_1) + K_{\rm eff}(x_1),a+a_{\rm eff})\,dx_1+m,
\end{equation}
where $Q_{\rm hom}$ (the \textit{homogenized elastic moduli}), $K_{\rm eff}$ (the \textit{spontaneous curvature--torsion tensor}), $a_{\rm eff}$ (the \textit{spontaneous infinitesimal stretch}), and $m\geq 0$ (the \textit{incompatibility of the prestrain}) are quantities that only depend on the linearized material law $Q$, the prestrain $B$, the geometry of the cross-section $S$, and the scale ratio $\gamma$; in particular,
\begin{itemize}
\item $m\geq 0$ is a constant given in Definition~\ref{D:eff} below,
\item $Q_{\rm hom}:\omega\times\Skew 3\times \R\to\R$ is a positive--definite quadratic form given by the homogenization formula of Definition~\ref{def:Qgamma} below,
\item $K_{\rm eff}\in L^2(\omega;\Skew 3)$ and $a_{\rm eff}\in L^2(\omega)$ are given by the averaging formula of Definition~\ref{D:eff} below.
\end{itemize}
Our main result establishes $\Gamma$-convergence of $\mathcal I^{\e(h),h}$ to $\mathcal I$:
\begin{theorem}[$\Gamma$-convergence]\label{Th:gamma0}
  Suppose Assumptions~\ref{ass:W} -- \ref{ass:B} are satisfied. For $u\in H^1(\Omega;\R^3)$ denote by
    \begin{equation}\label{def:Eh}
     \Eop_h(u):=\frac{\sqrt{\nabla_h u^t\nabla_h u}-\Id}{h}.
    \end{equation}
    the (scaled) nonlinear strain tensor. Then:% there exist $K_{B,\gamma}\in L^2(\omega;\Skew 3)$ and $E_{B,\gamma}\in[0,\infty)$ given in Definition~\ref{def:KE} such that the following is true:
  \begin{itemize}
  \item[(a)] (Compactness). Let $(u^h)\subset L^2(\Omega;\R^3)$ be a sequences with equibounded energy, i.e.~
    \begin{equation}\label{eq:equibounded}
      \limsup_{h\to0}\mathcal I^{\e(h),h}(u^h)<\infty.
    \end{equation}
    Then there exists $(u,R,a)\in\mathcal A$ and a subsequence (not relabeled) such that
    \begin{eqnarray}\label{eq:convergence1}
      (u^h-\fint_{\Omega}u^h,\nabla_h u^h)\to& (u,R)\qquad&\text{in }L^2(\Omega)\\
      \fint_S\Eop_h(u^h)\cdot (e_1\otimes e_1)  \wto& a\qquad &\text{in }L^2(\omega).\label{eq:convergence2}
    \end{eqnarray}
  \item[(b)] (Lower bound). Let $(u^h)\subset L^2(\Omega;\R^3)$ be a sequence that converges to some $(u,R,a)\in\mathcal A$ in the sense of \eqref{eq:convergence1} and \eqref{eq:convergence2}. Then
    \begin{equation*}
      \liminf_{h\to0} \mathcal I^{\e(h),h}(u^h)\geq \mathcal I(u,R,a).
    \end{equation*}
  \item[(c)] (Recovery sequence). For any $(u,R,a)\in\mathcal A$ there exists a sequence $(u^h)\subset L^2(\Omega;\R^3)$ converging to $(u,R,a)$ in the sense of in the sense of \eqref{eq:convergence1} and \eqref{eq:convergence2} such that
    \begin{equation*}
      \lim_{h\to0}\mathcal I^{\e(h),h}(u^h)= \mathcal I(u,R,a).
    \end{equation*}
  \end{itemize}
\end{theorem}
\proofref{sec:proofmain}

\begin{remark}\label{rem:purebendingtorsion}
Theorem~\ref{Th:gamma0} also yields a compactness and $\Gamma$-convergence result towards a (more conventional) pure bending--torsion model. Indeed, by part (a) of Theorem~\ref{Th:gamma0} every sequence with equibounded energy satisfies \eqref{eq:convergence1} for some rod-deformation $(u,R)$ satisfying \eqref{def:A}. Furthermore, by minimizing over $a\in L^2(\omega)$ the statements of the parts (b) and (c) in Theorem~\ref{Th:gamma0} hold with $(u,R,a)$ and $\mathcal I(u,R,a)$ replaced by $(u,R)$ and  ${\mathcal I}'(u,R):=\inf_{a\in L^2(\omega)}\mathcal I(u,R,a)$ (see Remark~\ref{rem:Qwithouta} below for a more explicit characterization of $\mathcal I'$). 
\end{remark}
%
%%
%\begin{remark}\label{rem:varyingS}
%Energies of the type \eqref{def:ene} naturally emerge in models of rods with varying cross-section. We describe this in the simple set-up of an homogeneous material law that occupies a cylindrical domain with a rapidly oscillating cross-section  $S_\e(x_1)\subset\R^2$: For $x_1\in (0,\ell)$ set $S_\e(x_1):=(1+ g_\e(x_1))S$, where $S\subset \R^2$ is a Lipschitz domain satisfying \eqref{ass:S} and $g_\e:\R\to(-1,\infty)$ is sufficiently smooth, and consider the reference domain
%\begin{equation*}
%  \Omega_\e:=\bigcup_{x_1\in(0,\ell)}\{x_1\}\times S_\e(x_1)\subset\R^3.
%\end{equation*}
%%
%Then,
%%
%\begin{align}\label{ene:rewritten}
% \frac1{h^2}\int_{\Omega_\e} W(\nabla_h u(x))\,dx=\frac1{h^2}\int_\Omega \widetilde W_\e(x ,\nabla_h \tilde u_\e(x)(\Id +h \widetilde B_{\e,h}(x))^{-1})\,dx,
%\end{align}
%%
%where $\Omega=(0,\ell)\times S$, $\tilde u_\e(x):=u(x_1,(1+g_\e(x_1))\bar x)$ and
%%
%\begin{align*}
% \widetilde W_\e(x,F):=&W(F)(1+g_\e(x_1))^2,\quad
% \widetilde B_{\e,h}(x):=g_\e'(x_1)\begin{pmatrix}0&0&0\\  x_2&0&0\\ x_3&0&0\end{pmatrix}+\frac{g_\e(x_1)}h\begin{pmatrix}0&0&0\\0&1&0\\0&0&1\end{pmatrix}.
%\end{align*}
%%
%For $g_\e(x_1)=\e \hat g(\frac{x_1}\e)$ and $\hat g$ periodic, the a asymptotic behavior of \eqref{ene:rewritten} can be analyzed by the asymptotic behavior of functionals of the form \eqref{def:ene} provided $\gamma\in(0,\infty]$. 
%\end{remark}
%%
\subsection{Homogenization- and averaging formulas.} \label{sec:2.4}
The definitions of $Q_{\rm hom}$, $K_{\rm eff}$, and $m$ rely on the two-scale structure of limiting strains. To motivate the upcoming formulas, we recall a two-scale compactness statement for the nonlinear strain, see \cite[Theorem 3.5]{N12} (and also Proposition~\ref{P:comact2} below): Suppose $(u_h)$ is a sequence with equibounded energy (cf.~\eqref{eq:equibounded}) with limit $(u,R,a)\in\mathcal A$ (cf.~\eqref{eq:convergence1}, \eqref{eq:convergence2}), then  (up to a subsequence) the associated scaled nonlinear strain tensors $\Eop_h(u_h)$ weakly two-scale converge in $L^2$ to a limiting strain $E:\Omega\times\mathcal Y\to\Sym 3$ of the form
\begin{equation}\label{strain-decomposition}
  E=\underbrace{\sym\left[R^t\partial_1R\, (\bfx\otimes e_1)\right]}_{\text{bending/torsion}}+\underbrace{a(e_1\otimes e_1)}_{\text{infinitesimal stretch}}+\underbrace{\chi.}_{\text{corrector field}}
\end{equation}
Above, $\chi\in L^2(\omega;\H^\gamma_{\rm rel})$ where ${\H}_{\operatorname{rel}}^\gamma\subset L^2(S\times  \mathcal Y;\Sym 3)$ is defined as follows:
\begin{equation}\label{def:Hgammarel}
  \begin{aligned}
    &\text{for $\gamma=0$:}&&\bigg\{\sym\left[\left((\partial_y\Psi)\bfx + \partial_y \hat \phi\,|\,\bar \nabla \bar \phi\right)\right]\,:\,\\
    &&&\qquad\Psi\in H^1(\mathcal Y;\Skew 3),\,\hat\phi\in H^1(\mathcal Y;\R^3),\,\bar \phi\in L^2(\mathcal Y;H^1(S;\R^3))\bigg\},\\
    &\text{for $\gamma=\infty$:}&&\left\{\sym\left[\left(\partial_y \hat \phi\,|\,\bar \nabla \bar\phi\right)\right]\,:\,\hat\phi\in L^2(S;H^1(\mathcal Y;\R^3)),\,\bar \phi\in H^1(S;\R^3)\right\},\\
    &\text{for $\gamma\in(0,\infty)$:}&&\left\{\sym\left[\left(\partial_y \phi\,|\,\tfrac1\gamma \bar \nabla \phi\right)\right]\,:\,\phi\in H^1(S\times \mathcal Y;\R^3)\right\}.
  \end{aligned}
\end{equation}
Note that on the right-hand side in \eqref{strain-decomposition} the first and second term are determined by the limiting deformation $(u,R,a)$. Only the third term $\chi$---the only term that involves the fast variable $y\in\mathcal Y$---depends on the chosen subsequence. We call it the \textit{strain corrector}. For the following discussion it is convenient to define for $(K,a)\in\Skew 3\times\R$ the affine map
\begin{equation}\label{eq:def:limit-strain}
  \Eop(K,a):S\to\Sym 3,\qquad 
  \Eop(K,a):=\sym\big[(K\bfx+a e_1)\otimes e_1\big],
\end{equation}
and to introduce the \textit{two-scale strain space}
\begin{equation}\label{def:Hgamma}
  \H^\gamma:=\Big\{{\Eop}(K,a)+\chi\,|\,(K,a)\in\Skew 3\times\R,\,\chi\in\H_{\rm rel}^\gamma\,\Big\}.
\end{equation}
Since $R^t\partial_1R$ is skew-symmetric (almost surely) for $(u,R,a)\in\mathcal A$, the limiting strain of \eqref{strain-decomposition} satisfies $E\in L^2(\omega;\H^\gamma)$.

\smallskip
\paragraph{Formula for $Q_{\rm hom}$.} As in \cite{N12} the homogenized quadratic form $Q_{\rm hom}$ is defined by minimizing out the energy contribution coming from $\chi\in\H^\gamma_{\rm rel}$:
\begin{definition}[homogenization formula for $Q_{\hom}$]\label{def:Qgamma}
  We define $Q_{\rm hom}:\omega\times \operatorname{Skew}(3)\times\R\to[0,\infty)$ by 
\begin{equation}\label{D:hom}
  Q_{\rm hom}(x_1,K,a):=\inf_{\chi\in \H^\gamma_{\rm rel}}\iint_{S\times Y}Q(x_1,\bar x,y,{\Eop}(K,a)+\chi(\bar x,y))\,dy\,d\bar x.
\end{equation}
\end{definition}
\begin{remark}
 We emphasize that the definition of $Q_{\rm hom}$ depends on the small-scale coupling $\gamma$ via the relaxation space $\H^\gamma_{\rm rel}$.% For almost every $x_1\in\omega$, $Q_{\rm hom}(x_1,\cdot)$ defines a positive definite quadratic form on $\Skew 3\times \R$ and the map $x_1\mapsto Q(x_1,K,a)$ is piecewise continuous for every $(K,a)\in \Skew 3\times \R$ (see Proposition~\ref{P:BVP} below). 
\end{remark}

\begin{remark}\label{rem:Qwithouta}
As already discussed in Remark~\ref{rem:purebendingtorsion}, a pure bending--torsion model is obtained from $\mathcal I$ by minimizing out the stretch variable $a$. This can be made more explicit as follows:
\begin{equation}\label{def:eneIrelaxa}
 \mathcal I'(u,R):=\inf_{a\in L^2(\omega)}\mathcal I(u,R,a)=\int_\omega  {Q_{\rm hom}'}(x_1,R^t(x_1)\partial_1R(x_1) + K_{\rm eff}(x_1))\,dx_1+m,
\end{equation}
where
\begin{equation*}
 Q_{\rm hom}'(x_1,K):=\inf_{\chi\in \H^\gamma_{\rm rel}\atop a\in\R}\iint_{S\times Y}Q(x_1,\bar x,y,{\Eop}(K,a)+\chi(\bar x,y))\,dy\,d\bar x.
\end{equation*}
The quadratic form $Q_{\rm hom}'$ coincides with the homogenized quadratic form given in \cite{N12} where the case without prestrain is studied.
\end{remark}
\paragraph{Formulas for $K_{\rm eff}$ and $a_{\rm eff}$.} We first present a ``geometric'' definition---an alternative ``algorithmic'' definition that is more practical for numerical investigations is presented in Section~\ref{S:algo} below. The geometric definition invokes the following Hilbert-space structure on $\H:=L^2(S\times\mathcal Y;\Sym 3)$: Let $\mathbb L$ denote the symmetric fourth-order tensor obtained from the quadratic form $Q$ by polarization, and consider for $x_1\in\omega$,
\begin{equation*}%\label{def:scalarproduct}
  \Big(F,G\Big)_{x_1}:=\iint_{S\times\mathcal Y}\left\langle\mathbb L(x_1,\bar x,y)F(\bar x,y),G(\bar x,y)\right\rangle\,dyd\bar x,\qquad F,G\in \H.
\end{equation*}
Since $Q$ is positive-definite and bounded on symmetric matrices, $(\cdot,\cdot)_{x_1}$ defines a scalar product on $\H$. We write $\|\cdot\|_{x_1}$ for the associated norm and note that
\begin{equation}\label{equivalence}
\sqrt\alpha\|\cdot\|_{L^2(S\times Y)}\leq  \|\cdot\|_{x_1}\leq\sqrt\beta\|\cdot\|_{L^2(S\times Y)}\qquad\text{on }\H.
\end{equation}
$\H^\gamma_{\rm rel}$ and $\H^\gamma$ (see \eqref{def:Hgammarel} and \eqref{def:Hgamma}) are closed, linear subspaces of $(\H,\|\cdot\|_{x_1})$. We denote by $(\H^{\gamma})^{\perp x_1}\subset\H$ (resp. $(\H^\gamma_{\rm rel})^{\perp,x_1}\subset \H^\gamma$) the $(\cdot,\cdot)_{x_1}$-orthogonal complement of $\H^\gamma$ in $\H$ (resp.~$\H^\gamma_{\rm rel}$ in $\H^\gamma$), and by $P^{\gamma,x_1}:\H\to(\H^\gamma)^{\perp x_1}$ and $P^{\gamma,x_1}_{\rm rel}:\H\to (\H^\gamma_{\rm rel})^{\perp x_1}$ the associated $(\cdot,\cdot)_{x_1}$-orthogonal projections. We thus have the orthogonal decomposition,
\begin{equation}\label{decomposition}
  \begin{aligned}
    \H\,=&\,\H^\gamma\oplus (\H^\gamma)^{\perp x_1}=\H^\gamma_{\rm rel}\oplus (\H^\gamma_{\rm rel})^{\perp x_1}\oplus (\H^\gamma)^{\perp x_1}\\
    =&\,\H^\gamma_{\rm rel}\oplus \mbox{range}(P^{\gamma,x_1}_{\rm
      rel})\oplus \mbox{range}(P^{\gamma,x_1}).
  \end{aligned}
\end{equation}
A direct consequence is the following observation:
\begin{lemma}[Pythagoras]\label{L:pythagoras}
  For all $x_1\in\omega$ and $F\in\H$,
  \begin{equation}\label{eq:pythagoras}
    \inf_{\chi\in\H^\gamma_{\rm rel}}\big(F+\chi,F+\chi)_{x_1}=\|P^{\gamma,x_1}F\|_{x_1}^2+\|P^{\gamma,x_1}_{\rm rel}F\|_{x_1}^2.
  \end{equation}
\end{lemma}
In particular, we obtain the following characterization of $Q_{\rm hom}$:
\begin{equation}\label{eq:qhomproj}
  Q_{\hom}(x_1,K,a)=\|P^{\gamma,x_1}_{\rm rel}\big({\Eop}(K,a)\big)\|_{x_1}^2.
\end{equation}
It turns out that any $E\in(\H^\gamma_{\rm rel})^{\perp x_1}$ admits a representation via a unique pair $(K,a)\in\Skew 3\times\R$:
\begin{lemma}[Representation]\label{L:represent}
  For all $x_1\in\omega$ the map
  \begin{equation*}
    {\Eop}^{\gamma,x_1}:\Skew3\times\R\to(\H^\gamma_{\rm rel})^{\perp x_1},\qquad {\Eop}^{\gamma,x_1}(K,a):=P^{\gamma,x_1}_{\rm rel}\big({\Eop}(K,a)\big)
  \end{equation*}
  defines a linear isomorphism and there exists a constant $C=C(\alpha,\beta,\gamma,S)$ such that
  \begin{equation}\label{estimate}
    \frac{1}{C}|(K,a)|\leq \|{\Eop}^{\gamma,x_1}(K,a)\|_{L^2(S\times Y)}\leq C|(K,a)|.
  \end{equation}
\end{lemma}
\proofref{sec:proofbvp}

We denote by $P^{\gamma,\bullet}$ the unique bounded operator on $L^2(\omega;\H)$ defined by the identity
\begin{equation*}
  (P^{\gamma,\bullet}(\zeta H))(x_1)=\zeta(x_1)P^{\gamma,x_1}H\qquad\text{for all }\zeta\in C^\infty_c(\omega),\ H\in \H\text{ and }x_1\in\omega;
\end{equation*}
and define $P^{\gamma,\bullet}_{\rm rel}$ and $\Eop^{\gamma,\bullet}$ analogously. We are now in position to define $(K_{\rm eff}(x_1),a_{\rm eff}(x_1))$ and $m$:
\begin{definition}[averaging formula for $m$ and $(K_{\rm eff},a_{\rm eff})$]\label{D:eff}
  We set
  \begin{equation}\label{def:m}
    m:=\int_\omega\|P^{\gamma,x_1}\sym B(x_1,\cdot)\|^2_{x_1}\,dx_1,
  \end{equation}
  and define $(K_{\rm eff},a_{\rm eff})\in L^2(\omega;\Skew 3\times\R)$ as the unique field such that
  \begin{equation}\label{def:Keff}
    (K_{\rm eff},a_{\rm eff})=\big({\Eop}^{\gamma, \bullet}\big)^{-1}(P^{\gamma,\bullet}_{\rm rel}\sym B).
  \end{equation}
\end{definition}
%

% %
\section{Evaluation of the homogenization formulas via BVPs}\label{S:algo}
The definitions of $Q_{\rm hom}$, $K_{\rm eff}$ and $a_{\rm eff}$ (see Definitions~\ref{def:Qgamma}~and~\ref{D:eff}) are rather abstract. In this section we present a characterization that replaces the ``abstract'' operator in these definitions by boundary value problems for the system of linear elasticity on the domain $S\times Y$. To benefit from the  linearity of the map $(K,a)\mapsto (K_{\rm eff},a_{\rm eff})$, we set
\begin{equation}\label{def:K}
  K^{(2)}:=\frac12(e_2\otimes e_1-e_1\otimes e_2),\quad   K^{(3)}:=\frac12(e_3\otimes e_1-e_1\otimes e_3),\quad K^{(4)}:=\frac12(e_3\otimes e_2-e_2\otimes e_3),
\end{equation}
and note that this defines an orthonormal basis of $\Skew 3$. Moreover, we introduce the maps ${E}^{(i)}:S\to\Sym 3$, 
\begin{equation*}
  {E}^{(i)}:=
  \begin{cases}
    {\Eop}(0,1)&i=1,\\
    {\Eop}(K^{(i)},0)&i=2,3,4,
  \end{cases}
\end{equation*}
see \eqref{eq:def:limit-strain} for the definition of ${\Eop}$. Note that $\{E^{(i)}\,:\,i=1,\ldots,4\}$ spans the macroscopic strain space. In particular,  $E^{(1)}$ corresponds to an infinitesimal stretch (in tangential direction); $E^{(i)}$ ($i=2,3$) corresponds to bending in direction $x_i$, and $E^{(4)}$ corresponds to a twist.

\smallskip

We have the following scheme to evaluate the homogenized quantities:
\begin{proposition}\label{P:BVP}
For $x_1\in\omega$ we define the following objects:
\begin{enumerate}[(i)]
\item The  \emph{strain correctors} $\chi^{(i)}(x_1)$ ($i=1,\ldots,4$) as the unique solution in $\H^\gamma_{\rm rel}$ to 
\begin{equation}\label{eq:cell1}
 \big(E^{(i)}+\chi^{(i)}(x_1),\chi\big)_{x_1}=0\qquad\text{for all }\chi\in\H^\gamma_{\rm rel}.
\end{equation}
\item The \emph{averaging matrix} $\mathbb M(x_1)\in \Sym 4$ as the unique matrix with entries
\begin{equation}\label{def:M}
 \mathbb M(x_1)_{ij}:=\big(E^{(i)}+\chi^{(i)}(x_1),E^{(j)}+\chi^{(j)}(x_1)\big)_{x_1}.
\end{equation}
\item The \emph{vector representation of the strain} $b(x_1)\in\R^4$ as the unique vector with entries
\begin{equation*}
 b(x_1)_{i}:=\big(B(x_1),E^{(i)}+\chi^{(i)}(x_1)\big)_{x_1}.
\end{equation*}
where $B$ denotes the prestrain tensor of Assumption~\ref{ass:B}.
\end{enumerate}
Then:
\begin{enumerate}[(a)]
\item $\mathbb M(x_1)$ is symmetric positive definite and we have $\frac{1}{C}\leq \mathbb M(x_1)\leq C$ (in the sense of quadratic forms) for a constant $C=C(\alpha,\beta,\gamma,S)$.
\item The map $x_1\mapsto \mathbb M(x_1)$ is piecewise continuous.
\item For all $(K,a)\in\Skew3\times\R$ we have
\begin{equation}\label{eq:qhomkmk}
 Q_{\hom}(x_1,K,a)=k\cdot \mathbb M(x_1)k\qquad\text{where }k:=\Big(a,(K\cdot K^{(2)}),\ldots, (K\cdot K^{(4)})\Big).
\end{equation}
\item With $k(x_1):=\mathbb M(x_1)^{-1}b(x_1)$ we have the identities
\begin{equation}\label{def:kinpbvp}
 K_{\rm eff}(x_1)=\sum_{i=2}^4k_i(x_1)K^{(i)},\qquad     a_{\rm eff}(x_1)=k_1(x_1)\qquad \text{for a.e. }x_1\in\omega.
\end{equation}
\end{enumerate}
\end{proposition}
\proofref{sec:proofbvp}
\smallskip

\begin{remark}[Averaging and homogenization] 
  The proposition shows that the spontaneous cur\-vature--torsion tensor
  $K_{\rm eff}$ and the spontaneous infinitesimal stretch $a_{\rm eff}$ linearly
  depend on $B$, and thus, the passage from $B$ to $(K_{\rm eff},a_{\rm eff})$
  can be interpreted as a \textit{spatial average} with a correction that takes
  the micro heterogeneity of the material, the cross-section $S$, and
  the scale ratio $\gamma$ into account. This is in contrast to the relation between $Q$ and
  $Q_{\rm hom}$, which is nonlinear and given by a \textit{homogenization} formula that has already been obtained in \cite{N12} where the case without prestrain is discussed. In \cite[Theorem~2]{KOB17}, a corresponding formula to \eqref{def:kinpbvp} is derived in the case of a homogeneous material law and a non-oscillatory prestrain. 
\end{remark}

Next, we derive boundary value problems (BVP) that allow to compute \eqref{eq:cell1}, and to represent the \textit{strain correctors} $\chi^{(i)}$.
\begin{lemma}[Characterization of the strain corrector via BVP]\label{L:BVP}
Fix $x_1\in\omega$ and $(a,K)\in \R\times\Skew 3$. Set $E:={\Eop}(K,a)$, and let $\chi_{E}\in\H^\gamma_{\rm rel}$ be the solution to
\begin{equation*}%\label{eq:cell}
  \big(E+\chi_E,\chi\big)_{x_1}=0\qquad\text{for all }\chi\in\H^\gamma_{\rm rel}.
\end{equation*}
\begin{enumerate}[(a)]
 \item Let $\gamma\in(0,\infty)$. Set $D^\gamma:=\mbox{\rm diag}(1,\gamma^{-1},\gamma^{-1})\in\Sym3$ and $\nabla:=(\partial_y,\bar\nabla)$. Then $\chi_E=\sym\left[D^\gamma\nabla\phi_E\right]$, where $\phi_E\in H^1(S\times\mathcal Y;\R^3)$ denotes the unique solution to
 \begin{equation*}
 \begin{aligned}
   &\iint_{S\times Y}\langle \mathbb L(x_1,\bar x,y)(E+D^\gamma\nabla\phi_E),D^\gamma\nabla\phi\rangle=0\qquad\text{for all }\phi\in H^1(S\times\mathcal Y;\R^3),
 \end{aligned}
 \end{equation*}    
 subject to
 \begin{equation}\label{eq:sidecondition01}
   \iint_{S\times Y}\phi_E=0\qquad \text{and}\qquad \iint_{S\times Y}\phi_E\cdot\bfx=0.
 \end{equation}
 \item Let $\gamma=0$. Consider the Hilbert space
  \begin{eqnarray*}
    {\bf X}^0&:=&\Big\{(\Psi,\hat\phi,\bar\phi)\in H^1(\mathcal Y;\Skew 3)\times H^1(\mathcal Y)\times L^2(\mathcal Y;H^1(S;\R^3)),\\
      &&\qquad\qquad\text{such that }\int_{Y}\Psi=0,\quad\int_Y\hat\phi=0,\quad\int_Y|\int_S\bar\phi|=0\,\Big\}.
   \end{eqnarray*}
   Then the map
    \begin{equation*}
      \iota^0:{\bf X}^0\to\H^\gamma_{\rm rel},\qquad (\Psi,\hat\phi,\bar\phi)\mapsto {\Eop}(\partial_y\Psi, \partial_y\hat\phi)+
      \sym\left[\left(0\,|\,\bar \nabla \bar \phi\right)\right]
    \end{equation*}
    defines an isomorphism, and we have $\chi_E=\iota^0(\Psi_E,\hat\phi_E,\bar\phi_E)$, where $(\Psi_E,\hat\phi_E,\bar\phi_E)\in{\bf X}^0$ denotes the unique solution to
    \begin{equation*}
      \begin{aligned}
        &\iint_{S\times Y}\langle \mathbb L(x_1,\bar
        x,y)\big(E+\iota^0(\Psi_E,\hat\phi_E,\bar\phi_E)\big),\,\iota^0(\Psi,\hat\phi,\bar\phi)\rangle=0\\
        &\qquad\text{for all }(\Psi,\hat\phi,\bar\phi)\in{\bf X}^0.
      \end{aligned}
    \end{equation*}    
  \item Let $\gamma=\infty$. Consider the Hilbert space
    \begin{eqnarray*}
      {\bf X}^\infty&:=&\Big\{(\hat\phi,\bar\phi)\in L^2(S;H^1(\mathcal Y;\R^3))\times H^1(S;\R^3),\\
      &&\qquad\qquad\text{such that }\int_S|\int_Y\hat\phi|=0,\quad\int_S\bar\phi=0\,\Big\}.
    \end{eqnarray*}
    Then the map
    \begin{equation*}
      \iota^\infty:{\bf X}^\infty\to\H^\infty_{\rm rel},\qquad (\hat\phi,\bar\phi)\mapsto
      \sym\left(\partial_y \hat \phi\,|\,\bar \nabla \bar \phi\right).
    \end{equation*}
    defines an isomorphism, and we have $\chi_E=\iota^\infty(\hat\phi_E,\bar\phi_E)$, where $(\hat\phi_E,\bar\phi_E)\in{\bf X}^\infty$ denotes the unique solution to
    \begin{equation*}
      \begin{aligned}
        &\iint_{S\times Y}\langle \mathbb L(x_1,\bar
        x,y)\big(E+\iota^\infty(\hat\phi_E,\bar\phi_E)\big),\,\iota^\infty(\hat\phi,\bar\phi)\rangle=0\\
        &\qquad\text{for all }(\hat\phi,\bar\phi)\in{\bf X}^\infty.
      \end{aligned}
    \end{equation*}    
  \end{enumerate}
\end{lemma}
\proofref{sec:proofbvp}

\section{Examples and explicit formulas for isotropic materials}\label{s:isotropic}
In this section we restrict our analysis to \textit{isotropic} materials and the extreme regimes $h\ll\e$ and $\e\ll h$, i.e., $\gamma\in\{0,\infty\}$. In that case the homogenized quantities---the matrix $\mathbb M$ from Proposition~\ref{P:BVP}---can be computed by hand, see Lemma~\ref{L:isotropic} below. We further specify the  findings of Lemma~\ref{L:isotropic} in the case of a \textit{bilayer material} which was studied in the homogeneous case in \cite{CRS,KOB17}. We observe a dramatic \textit{size effect}: We give an explicit example of a prestrain $B$ that produces zero spontaneous bending in the case $\gamma=0$ but non-zero bending in the case $\gamma=\infty$. Moreover, we apply Lemma~\ref{L:isotropic} to prestrain tensors that originate from models for nematic liquid crystal elastomers and compare the results with the findings of \cite{ADeS16,ADeSK17} in the context of ribbons. Finally, we address shape programming. 
\medskip

\subsection{Isotropic, laterally periodic composites.} 
Throughout this section we suppose that the composite is isotropic, periodically oscillating in longitudinal direction, and constant in cross-sectional direction, i.e., we suppose that  $Q$ (cf.~Assumption~\ref{ass:W}) is of the form 
\begin{equation}\label{Q:isotropic}
 Q(x,y,G)=Q(y,G)=2\mu(y)|\sym G|^2+\lambda(y)(\operatorname{trace}\ G)^2,
\end{equation}
with (periodic) Lam\'e-constants $\mu,\lambda\in L^\infty(\mathcal Y)$ that are (essentially) non-negative, and $\essinf_{y\in\R}(2\mu+\lambda)>0$. We recall the definition of some standard moduli for isotropic elastic materials:
\begin{eqnarray*}
 \nu&:=&\frac{\lambda}{2(\mu+\lambda)}\qquad\text{(Poisson ratio)},\\
 \beta&:=&2\mu+\lambda-2\lambda\nu=\frac{\mu(2\mu+3\lambda)}{\mu+\lambda}\qquad\text{(Young's modulus)},\\
  M&:=&2\mu+\lambda\qquad\text{(P-wave modulus)}.
\end{eqnarray*}
The formulas for the elastic moduli of the effective model involve the arithmetic and harmonic mean. To shorten notation, for $f\in L^1(\mathcal Y)$ we set
\begin{equation*}
  \langle f\rangle:=\int_Yf(y)\,dy\qquad\text{and}\qquad \langle f\rangle_{\hom}:=\left(\int_Y\frac{1}{f(y)}\,dy\right)^{-1}.
\end{equation*}
Furthermore, we define the effective moduli
\begin{eqnarray*}
 \nu_\infty&:=&\frac12 \frac{\langle M\rangle_\ho \langle \frac{\lambda}{M}\rangle}{\langle\frac{\lambda}{M}\rangle^2\langle M\rangle_\ho -\langle\frac{\lambda^2}{M}\rangle + \langle \lambda+\mu\rangle},\\
  \beta_\gamma&:=&
  \begin{cases}
    \langle\beta\rangle_{\rm hom}&\gamma=0,\\
    \langle M\rangle_\ho(1-2\nu_\infty\langle \frac{\lambda}{M}\rangle)&\gamma=\infty,
  \end{cases}
\end{eqnarray*}
and note that $\nu_\infty=\nu$ and $\beta_\infty=\beta_0=\beta$ for homogeneous, isotropic materials. Next to the elastic moduli, the homogenized model depends on the geometry of the cross--section $S$. To capture this effect, we denote by $\varphi_S\in H^1(S)$ the unique minimizer to 
\begin{equation}\label{def:qs}
  \tau_S:=\min_{\varphi\in H^1(S)}\int_S \big((\partial_2 \varphi - x_3)^2+(\partial_3 \varphi + x_2 )^2\big)=\int_S \big((\partial_2 \varphi_S - x_3)^2+(\partial_3 \varphi_S + x_2 )^2\big),
\end{equation}
satisfying $\int_S\varphi_S=0$. Following \cite[Remark~3.5]{MM03}, we refer to the function $\varphi_S$ and the parameter $\tau_S$ as the \textit{torsion function} and the \textit{torsional rigidity}.

\medskip

The following lemma yields an explicit expression for the averaging matrix $\mathbb M$ of Proposition~\ref{P:BVP} in terms of averages of the Lam\'e-constants, the torsional rigidity and the torsion function. It can be seen as an extension of the analysis in \cite{MM03} and \cite[Theorem~3]{KOB17} to periodic composites and periodic prestrain.
\begin{lemma}[Effective properties in the isotropic case]\label{L:isotropic}
  Let $\gamma\in\{0,\infty\}$. Let $\mu,\lambda\in L^\infty(\mathcal Y)$ be non-negative and satisfy $\essinf_{y\in\R}(2\mu+\lambda)>0$. Suppose that $Q$ (cf.~Assumption~\ref{ass:W}) is of the form \eqref{Q:isotropic}. Then:
\begin{enumerate}[(i)]
 \item the strain correctors $\chi^{(i)}$, $i=1,\dots,4$ defined via \eqref{eq:cell1} satisfy 
 \begin{subequations}
   \begin{eqnarray}\label{corr4}
     E^{(1)}+\chi^{(1)}&=&
     \begin{cases}
       \frac{\langle\beta\rangle_{\hom}}{\beta}\operatorname{diag}(1,-\nu,-\nu)&\gamma=0,\\
       \operatorname{diag}\left(\frac{\beta_\infty}{M}+2\nu_\infty \frac{\lambda}{M},-\nu_\infty,-\nu_\infty\right)&\gamma=\infty,
     \end{cases}\\\label{corr1}
     E^{(i)}+\chi^{(i)}&=&-x_i\big(E^{(1)}+\chi^{(1)}\big)\qquad\text{for }i=2,3,\\
     E^{(4)}+\chi^{(4)}&=&\frac{\langle \mu\rangle_\ho}{\mu} \begin{pmatrix}\label{corrtorsion}
       0 & \frac12 (\partial_2 \varphi_S-x_3) & \frac12 (\partial_3 \varphi_S+x_2)\\ \frac12 (\partial_2 \varphi_S-x_3) &0&0 \\ \frac12 (\partial_3 \varphi_S + x_2)&0&0
     \end{pmatrix}.
   \end{eqnarray}
  \end{subequations}
  \item The matrix $\mathbb M$ of Proposition~\ref{P:BVP} satisfies
    \begin{equation*}
      \mathbb M=\mbox{\rm diag}\left(\beta_\gamma\int_S\,d\bar x,\ \ \beta_\gamma\int_Sx_2^2\,d\bar x,\ \ \beta_\gamma\int_Sx_3^2\,d\bar x,\ \ \langle\mu\rangle_{\rm hom}\tau_S\right).
    \end{equation*}
  \item The vector $b$ of Proposition~\ref{P:BVP} satisfies
    \begin{eqnarray*}
      \left(\begin{array}{c}
          b_1\\b_2\\b_3
        \end{array}\right)
      &=&\beta_\gamma\iint_{S\times Y}B_{11}(e_1-\bfx)
    \\&&-
    \begin{cases}\displaystyle
      0&\text{for }\gamma=0,\\\displaystyle
      \iint_{S\times Y}g_\infty(B_{22}+B_{33})(e_1-\bfx)&\text{for }\gamma=\infty
    \end{cases}\\
    b_4&=&\langle\mu\rangle_\ho\int_S (\partial_2\varphi_{S}-x_3)\,\langle B_{12}+B_{21}\rangle\,+(\partial_3\varphi_{S}+x_2)\,\langle B_{13}+B_{31}\rangle\,d\bar x,
    \end{eqnarray*}
    where 
    \begin{equation}\label{def:ginfty}
     g_\infty:= 2(\mu+\lambda)\nu_\infty-(\beta_\infty+2\nu_\infty\lambda)\frac{\lambda}M.
    \end{equation}
\item The vector $k$ of Proposition~\ref{P:BVP} is given by $k_i =\mathbb M_{ii}^{-1}b_i$ for $i=1,\dots,4$, and it holds
\begin{equation*}
 Q_\ho(K,a)=|S|\beta_\gamma a^2+\sum_{i=2}^3\beta_\gamma \int_S x_i^2\,d\bar x k_i^2+\langle \mu\rangle_\ho \tau_S \tau^2,
\end{equation*}

where $a\in\R$ and $K=\begin{pmatrix}0& k_2 & k_3\\-k_2&0&\tau\\-k_3&-\tau&0                                                                                                                                          \end{pmatrix}$.
\end{enumerate}
  
\end{lemma}
\proofref{sec:proofisotropic}

\begin{remark}[General observations]\label{R:isotropic}
The qualitative dependency of the spontaneous curvature--torsion tensor
\begin{equation*}
  k=(k_1,k_2,k_3,k_4)\, \widehat{=}\, \big(\text{infinitesimal stretch},\, \text{bend},\, \text{bend},\,\text{twist}\big)
\end{equation*}
 on the geometry of $S$, the prestrain $B$ and the material law can be summarized in the following diagram.%

\smallskip

\begin{center}
\begin{tabular}{|c|c|c|cl|}
\hline
 & Geometry $S$ & prestrain $B$ & material law&  \\
\hline

$\begin{pmatrix}k_1\\k_2\\k_3\end{pmatrix}$ & linear & $\begin{cases}\langle B_{11}\rangle&\mbox{if $\gamma=0$}\\ {\langle B_{11}\rangle, \langle g_\infty(B_{22}+B_{33})\rangle}&{\mbox{if $\gamma=\infty$}}\end{cases}$& $\begin{cases}{\boxtimes}&\mbox{if $\gamma=0$}\\ {\lambda,\mu}&{\mbox{if $\gamma=\infty$}}\end{cases}$&\\

\hline
$k_4$ &$\mbox{non-linear}$ & $e_1\cdot \langle \sym B\rangle e_j,\,j\in\{2,3\}$  &${\boxtimes} $  &\\

\hline
\end{tabular} 
\end{center}
In \cite[Theorem 3]{KOB17} the statement of Lemma~\ref{L:isotropic} is given in the case of a homogeneous material and non-oscillatory prestrain. The values for the induced torsion $k_4$ in the case $\gamma\in\{0,\infty\}$ and for the induced stretching and bending $k_1,k_2,k_3$ in the case $\gamma=0$ coincides with the findings of \cite{KOB17} applied to the averaged prestrain $\langle B\rangle\in L^2(S;\R^{3\times 3})$. In the case $\gamma=\infty$ the values of $k_1,k_2,k_3$ differ substantially from the homogeneous case. Finally, we note that $g_\infty$ given in \eqref{def:ginfty} satisfies $\langle g_\infty\rangle=0$. In particular, for homogeneous isotropic materials, i.e.\ $\lambda$ and $\mu$ are constant, the vector $k$ coincides in the cases $\gamma=0$ and $\gamma=\infty$.

\end{remark}

\subsection{Example 1: Isotropic bilayer with isotropic prestrain.}\label{sec:4.2}
Set $S=(-1,1)^2$ and choose $B=\operatorname{sgn}(x_3)\rho(y)\Id$ for some $\rho\in L^2(\mathcal Y)$. Then, it is easy to check that $b_2=b_4=0$ and thus $k_2=k_4=0$ and
\begin{align*}
 \gamma=0:\quad &k_3=\langle \rho\rangle \frac{-\int_S |x_3|}{\int_S x_3^2}=-\frac{3}{2}\langle \rho\rangle,\quad k_1=\langle \rho\rangle\\
 \gamma=\infty:\quad&k_3=-\frac{3}{2}\langle \rho\rangle + \langle 2\rho g_\infty\rangle \frac{\int_S |x_3|}{\beta_\infty \int_S x_3^2}=-\frac{3}{2}(\langle \rho\rangle -2\frac{\langle \rho g_\infty\rangle}{\beta_\infty} ),\quad k_1=\langle \rho\rangle -2 \frac{\langle \rho g_\infty\rangle}{\beta_\infty}.
\end{align*}
Next, we consider an isotropic two-phase composite where the first Lam\'e constant $\lambda$ is constant and the shear moduli oscillates. More precisely, for given $\vartheta\in[0,1]$ we set 
\begin{equation*}%\lambda(y)=\begin{cases}
             %\lambda_1&\mbox{if $y\in(0,\vartheta)$}\\ \lambda_2&\mbox{if $y\in(\vartheta,1)$}
            %\end{cases},\quad  
            \mu(y)=\begin{cases}
             \mu_1&\mbox{if $y\in(0,\vartheta)$}\\ \mu_2&\mbox{if $y\in(\vartheta,1)$}
            \end{cases},\quad \rho(\vartheta,y)=\begin{cases}
          1 &\mbox{for $y\in(0,\vartheta)$}\\-1&\mbox{for $y\in(\vartheta,1)$}
         \end{cases}
\end{equation*}
and $\lambda(y)=\lambda>0$ for all $y\in\mathcal Y$.

Note that in the case $\gamma=0$ the map induced infinitesimal stretch $k_1$ and bending $k_3$ are affine in the volume fraction $\vartheta$ and $\langle \rho\rangle =0$ (i.e.\ $\vartheta=\frac12$) implies $k=0$. In the case $\gamma=\infty$ the map $\vartheta\mapsto k(\vartheta)$ is in general non-linear and non-monotone (see Figure~\ref{figrefcelllayerd}). Indeed, we have 
\begin{align*}
 \nu_\infty=&\frac12 \frac{\lambda}{\lambda+\vartheta\mu_1+(1-\vartheta)\mu_2}=\frac{\lambda}{\lambda+\vartheta M_1+ (1-\vartheta)M_2}\\
 \beta_\infty=&\langle M\rangle_\ho -2\lambda\nu_\infty=\frac{M_1M_2}{\vartheta M_2+(1-\vartheta)M_1}-\frac{2\lambda^2}{\lambda+\vartheta M_1+ (1-\vartheta)M_2}\\
 \langle g_\infty \rho\rangle%=& 2\nu_\infty\left\langle \left(\lambda + \mu\right)\rho\right\rangle-\langle M\rangle_\ho \lambda\left\langle\frac{\rho}{M}\right\rangle\\
 =&\frac{\lambda(\vartheta(M_1+\lambda)-(1-\vartheta)(M_2+\lambda))}{\lambda+\vartheta M_1+ (1-\vartheta)M_2}-\lambda \frac{\vartheta M_2-(1-\vartheta)M_1}{\vartheta M_2+(1-\vartheta)M_1}.
 \end{align*}
Recall that for $\vartheta=\frac12$, we have $\langle B\rangle =0$, but for $\gamma=\infty$ and $M_1\neq M_2$,
\begin{align*}
k_3(\vartheta=\tfrac12)=3\frac{\langle g_\infty \rho\rangle}{\beta_\infty}=\lambda \frac{(M_1-M_2)(\lambda+M_1+M_2)}{(M_1+M_2)(\lambda+\frac12 M_1+ \frac12M_2)}\neq0.
\end{align*}
 %
% In Figure~\ref{fig}, we plot the map $\vartheta\mapsto k_3(\vartheta)$ for some specific choices for $\lambda,\mu_1$ and $\mu_2$, and observe the nonlinear (and nonmonotone) dependence of the volumefraction $\vartheta$.

\begin{figure}
\centering
\includegraphics[scale=0.4]{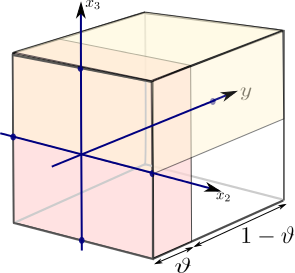}
  \caption{\small{Schematic picture of the reference cell in the example discussed in Section~\ref{sec:4.2}}}\label{figrefcelllayerd}
\end{figure}

 \begin{figure}
  \centering
  \includegraphics[width=0.9\textwidth]{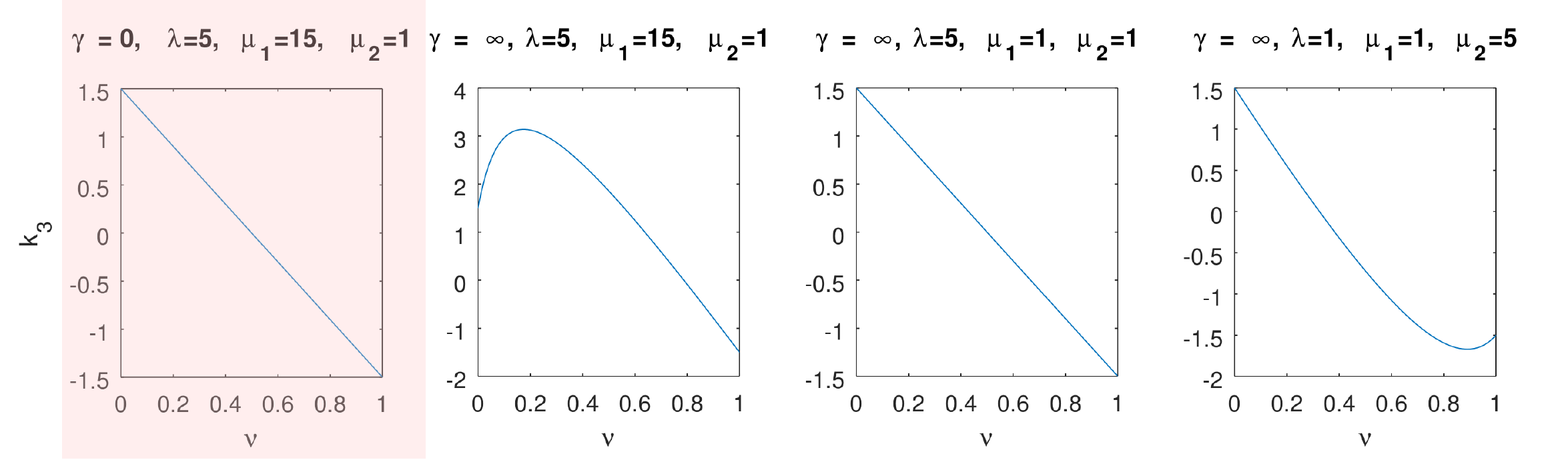}
  \caption{\small{Dependency of the spontaneous curvature $k_3$ (bending) on the volume fraction parameter $\vartheta$ in the example discussed in Section~\ref{sec:4.2}. The map $\vartheta\mapsto k_3(\vartheta)$ is linear if $\gamma=0$ and it might be nonlinear in the case $\gamma=\infty$ (depending on the parameters $\mu_1$ and $\mu_2$). We observe a size effect, e.g.\ from flat, $k_3=0$, to curved, $k_3\neq0$ by only changing $\gamma$ form $0$ to $\infty$.}}\label{plotslayered}
 \end{figure}

% \begin{figure}\label{fig}
%  \centering
%  \includegraphics[width=0.9\textwidth]{Plots.png}
%  \caption{\small{The plots show dependency of the spontaneous curvature $k_3$ on the parameter $\vartheta$ which describes the volume fraction of phase 1 of a layered two-phase composite with a layered microstructure and subject to an isotropic, bilayer prestrain in the case $\gamma=\infty$, see Section~\ref{sec:4.2}. The map $\vartheta\mapsto k_3(\vartheta)$ is plotted with $\lambda=5$, $\mu_1=15$, $\mu_2=1$ (left), $\lambda=5$, $\mu_1=\mu_2$ (middle), and $\lambda=1=\mu_1$, $\mu_2=5$ (right).}}
% \end{figure}

\subsection{Example 2: Nematic rods.}\label{par:liquid}
Liquid crystal elastomers are solids made of liquid crystals (rod--like molecules) incorporated into a polymer network. In a nematic phase (at low temperature) the liquid crystals show an orientational order and the material features a coupling between the entropic elasticity of the polymer network and the LC-orientation. The latter leads to a thermo-mechanical coupling that can be used in the design of active thin sheets that show a complex change of shape upon thermo-mechanical (or photo-mechanical) actuation, see \cite{B_White15}. Following \cite{Warner,BLP18} we describe the elastic energy of a nematic elastomer by the functional 
\begin{equation}\label{ene:lc:3d}
 \int_{\Omega} W(x,\nabla u(x) A^{-1}(x))\,dx,
\end{equation}
where the so-called step-length tensor is given by
\begin{equation}\label{A:lc}
 A:=r^{-1/3}(\Id + (r-1)n\otimes n).\qquad \mbox{with $r>1$}.
\end{equation}
Above $n:\Omega\to S^2:=\{x\in\R^3\,:\,|x|=1\}$ is a director field that describes the local orientation of the liquid crystals, and $r$ is a scalar order parameter. In \cite{BLS16} a non-Euclidean bending plate model is derived via $\Gamma$-convergence from \eqref{ene:lc:3d} under the assumption that the director field $n$ is sufficiently smooth and satisfies additional structural assumptions (in particular it is assumed to be constant in the thickness direction). In \cite{ADeS16}, the authors derive a plate model from the energy \eqref{ene:lc:3d} with director fields $n$ that are allowed to have large variations across the thickness but with the simplifying assumption that $r$ in \eqref{A:lc} is replaced by $r_h=1+\bar r h$ with $\bar r\in\R$, where $h$ denotes the thickness of the plate. Under this assumption, we have 
\begin{equation}\label{Ah:lc}
A_h^{-1}:=(r_h^{-1/3}(\Id + (r_h-1)n\otimes n))^{-1}=\Id+hB(n)+\mathcal O(h^2)\mbox{ with $B(n):=\frac13 \bar r \Id - \bar r n\otimes n$}.
\end{equation}
Two specific choices for the director field $n$ were studied in \cite{ADeS16,ADeSK17} in detail for the case of plates and ribbons: 
$$
\mbox{``Splay bend'':}\quad n^{\mathrm{(S)}}(x_3):=\begin{pmatrix}
                                                   \cos (\frac\pi4+ \frac{\pi}{4} x_3)\\0\\\sin (\frac\pi4+ \frac{\pi}{4} x_3)
                                                  \end{pmatrix}\qquad \mbox{``Twist'':}\quad n^{\mathrm{(T)}}(x_3):=\begin{pmatrix}
                                                   \cos (\frac\pi4+ \frac{\pi}{4} x_3)\\\sin (\frac\pi4+ \frac{\pi}{4} x_3)\\0
                                                  \end{pmatrix}.
$$
In the following, we present the spontaneous curvature--torsion vector $k=k(n)$ for prestrains $B(n)$ defined via \eqref{Ah:lc} with director fields $n$ corresponding to splay bend- and twist configurations. To be precise, set $S=(-1,1)^2$ and consider for simplicity the case of an isotropic and homogeneous material law, that is $Q$ (cf.~Assumption~\ref{ass:W}) is of the form $Q(x,y,G)=Q(G)=2\mu|\sym G|^2+\lambda(\operatorname{trace}\ G)^2$ with $\mu>0$ and $\lambda\geq0$. 
\begin{itemize}
 \item (splay bend). For given $\vartheta\in[0,\pi)$, set $n_\vartheta^{\mathrm{(S)}}(x_3):=(\cos (\vartheta+ \frac{\pi}{4} x_3),0,\sin (\vartheta + \frac{\pi}{4} x_3))^t$. Then,
 \begin{align*}
 k(n_\vartheta^{\mathrm{(S)}})=&\bar{r}(\tfrac5{12}+\tfrac{1}{2\pi}\cos(2\vartheta),0, \tfrac3{\pi^2}\sin(2\vartheta) ,0)^t.
 \end{align*} 
 \item (twist).  For given $\vartheta\in[0,\pi)$, set $n_\vartheta^{\mathrm{(T)}}(x_3)=(\cos (\vartheta+ \frac{\pi}{4} x_3),\sin (\vartheta + \frac{\pi}{4} x_3),0)^t$. Then,
 \begin{align*}
 k(n_\vartheta^{\mathrm{(T)}})=&k(n_\vartheta^{\mathrm{(S)}})+ \frac{c_S}{\tau_S} \bar r (0,0,0,\cos(2\vartheta))^t\qquad\mbox{where}\quad c_S:= \tfrac{8}{\pi^3}(8\tanh(\tfrac{\pi}{2})-\pi)>0.
 \end{align*} 
\end{itemize}
(For details on the calculations, we refer to Appendix~\ref{appendix:1})

Let us now compare the above findings with the results in \cite{ADeS16,ADeSK17}: In \cite{ADeS16} the authors derive a $2D$-plate model from the energy \eqref{ene:lc:3d} with $\vartheta=\frac{\pi}{4}$. Starting from the resulting plate model a $1D$-ribbon model is derived in \cite{ADeSK17} by cutting out a thin strip from the plate (in a certain angle $\theta$) and perform a dimension reduction limit similar to \cite{FHMP16a}. The limit model is based on a non-quadratic non strictly-convex function of bending and torsion. Hence, we cannot compare the results directly but at least, up to a non-vanishing prefactor, the preferred bending-torsion $(k_2,k_4)$ derived above lies in the set of preferred bending-torsion given by the model in \cite{ADeSK17}.

\subsection{Application: Shape programming via isotropic prestrain.}\label{sec:4.4}
In view of applications it is desirable to recover a given ``target'' spontaneous curvature--torsion tensor $K_{\rm eff}$ (cf.~Definition~\ref{D:eff}) by mixing simple microscopic building blocks that come in the form of parametrized microstructures. Recall that $K_{\rm eff}$ determines (up to an additive constant) the minimizer of the functional \eqref{def:eneIrelaxa} (and of \eqref{def:limitI} up to the infinitesimal stretch). Next, we show that a simple isotropic prestrain suffices to prescribe the bending part of $K_{\rm eff}$. To simplify the computations, we consider the following specific situation:
\begin{itemize}
\item The material is isotropic and homogeneous, i.e., we assume that $Q$ (cf.~Assumption~\ref{ass:W}) is of the form  $Q(x,y,G)=Q(G)=2\mu|G|^2+\lambda({\rm trace}\ G)^2$ with $\mu>0$ and $\lambda\geq 0$ being fixed from now on,
\item The cross-section of the rod is circular, i.e.\ $S:=\{\bar x\,:\,|\bar x|\leq 1\}$.
\item The prestrain tensor $B$ of Assumption~\eqref{ass:B} either vanishes or is equal to $\frac{3\pi}{8}\Id$, and the local prestrain microstructure is captured by the 2-parameter family
\begin{equation*}
  Z_{\theta,\alpha}:=\bigg\{(\bar x,y)\in S\times Y\,:\,y\in[0,\theta),\,\bar x=r\begin{pmatrix}\cos\varphi\\\sin\varphi\end{pmatrix}\,\mbox{with $r\in[0,1]$, $\varphi\in[\alpha-\frac{\pi}2,\alpha+\frac{\pi}2$)}\bigg\},
\end{equation*}
with angle $\alpha\in[0,2\pi)$ and volume-fraction $\theta\in [0,1]$, see Figure~\ref{referencecell-programming} for illustration. More precisely, we assume that for a.e.~$x_1\in\omega$ we have $\{B(x_1,\cdot)\neq 0\}=Z_{\theta(x_1),\alpha(x_1)}$ for suitable parameters $\alpha(x_1)$ and $\theta(x_1)$.
\item The relative scaling parameter satisfies $\gamma\in\{0,\infty\}$.
\end{itemize}
The upcoming result implies that any isometric curve $u\in W^{2,2}(\omega;\R^3)$ with $|u''|\leq 1$ can be recovered as a minimizer of a rod with a microstructured prestrain of the form
\begin{equation}\label{programm:B}
  B(x,y):=\frac{3\pi}{8}{\bf 1}_{Z_{\theta(x_1),\alpha(x_1)}}(\bar x,y)\Id,
\end{equation}
where $\alpha:\omega\to[0,2\pi)$ and $\theta:\omega\in[0,1]$ are suitable ``designs''. Since any isometric $u\in W^{2,2}(\omega;\R^3)$ is characterized (up to a rigid motion) by the first column of an associated $K\in L^2(\omega;\Skew 3)$, we only need to establish the following statement:
\begin{figure}
  \centering
  \includegraphics[scale=0.5]{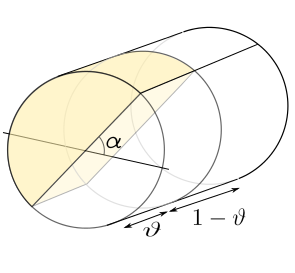}
  \caption{\small{Schematic picture of the building blocks. The set $Z_{\vartheta,\alpha}$ is highlighted in yellow.}}\label{referencecell-programming}
 \end{figure}
\begin{lemma}
To any $K\in L^2(\omega;\Skew 3)$ satisfying
\begin{equation}\label{ass:Knotwist}
  |Ke_1|\leq 1\quad \mbox{a.e.~in $\omega$},
\end{equation}
we can find ``designs''  $\alpha:\omega\to[0,2\pi)$ and $\theta:\omega\in[0,1]$ such that the spontaneous curvature--torsion tensor $K_{\rm eff}$ of Theorem~\ref{Th:gamma0} associated with the prestrain tensor $B$ defined in \eqref{programm:B} satisfies
\begin{equation*}
  (K-K_{\rm eff})e_1=0\qquad\text{a.e.~in }\omega.
\end{equation*}
\end{lemma}

\begin{proof}
Lemma~\ref{L:isotropic} and Remark~\ref{R:isotropic} yield
\begin{equation*}
 \mathbb M=\beta_\ho{\rm diag}\biggl(\tfrac\pi2,\tfrac\pi8,\tfrac\pi8,\frac{\langle\mu\rangle_\ho}{\beta_\ho}\tau_S\biggr),\quad b=\beta_\ho \theta \frac{3\pi}8\begin{pmatrix}\frac\pi2\\-\frac13\cos\alpha\\-\frac13\sin\alpha\\0\end{pmatrix}%\quad k= \theta\begin{pmatrix}\frac\pi6\\-\cos\alpha\\-\sin\alpha\\0\end{pmatrix}
\end{equation*}
and thus 
$$K_{\rm eff}(x_1)=-\theta(x_1)\cos(\alpha(x_1))K^{(2)}-\theta(x_1)\sin(\alpha(x_1))K^{(3)},$$
where $K^{(2)}$ and $K^{(3)}$ are defined in \eqref{def:K}. The claim follows.
\end{proof}

\section{Proofs}\label{S:proofs}

\subsection{Main result -- Proof of Theorem~\ref{Th:gamma0}}\label{sec:proofmain}

We first state a compactness and approximation result which is a simple consequence of \cite[Proposition~3.2]{N12} and \cite[Theorem~3.5]{N12}.
\begin{proposition}\label{P:comact2}
\begin{enumerate}[(a)]
\item (Compactness and identification). Suppose that $(u^h)\subset L^2(\Omega;\R^3)$ satisfies
\begin{equation}\label{est:finitebending}
 \limsup_{h\to0}\frac1{h^2}\int_\Omega \dist^2(\nabla_h u^h,\SO 3)\,dx<\infty.
\end{equation}
Then there exist $(u,R,a)\in\mathcal A$, $\chi\in L^2(\omega;\H_{\rm rel}^\gamma)$ and a subsequence (not relabeled) satisfying \eqref{eq:convergence1}, \eqref{eq:convergence2}, and
\begin{equation}\label{lim:compact2}
   \Eop_h(u^h)\stackrel{2,\gamma}{\wto} \Eop(R^t\partial_1R,a)+\chi\qquad\mbox{weakly two-scale in $L^2$,}
\end{equation}
where $\Eop_h$ and $\Eop$ are defined in \eqref{def:Eh} and \eqref{eq:def:limit-strain}, respectively.
\item (Approximation). For all $(u,R,a)\in\mathcal A$ and $\chi\in \H_{\rm rel}^\gamma$ there exists a sequence $(u^h)\subset L^2(\Omega;\R^3)$ satisfying \eqref{eq:convergence1}, \eqref{eq:convergence2}, and
\begin{eqnarray*}
 & &\Eop_h(u^h)\stackrel{2,\gamma}{\to} \Eop(R^t\partial_1R,a)+\chi\quad\mbox{strongly two-scale in $L^2$,}\\
 & &\limsup_{h\to0}\esssup_{x\in\Omega}\sqrt h\left(\frac{\dist(\nabla_h u^h(x),\SO 3)}{h}+|\operatorname \Eop_h(u^h)|\right)=0.
\end{eqnarray*}
\end{enumerate}
\end{proposition}

\begin{proof}[Proof of Proposition~\ref{P:comact2}]
 \step{1} Proof of (a).
 
 By \cite[Proposition~3.2]{N12} there exist $(u,R)\in\mathcal A$ and a subsequence (not relabeled) satisfying \eqref{eq:convergence1}. Moreover, by \cite[Theorem~3.5]{N12} there exists $a\in L^2(\omega)$ and $\chi\in\H_{\rm rel}^\gamma$ such that, up to extracting a further subsequence (not relabeled), we have 
\begin{equation}\label{lim:compact2b}
\Eop_h(u^h)\stackrel{2,\gamma}{\wto} \Eop(R^t\partial_1R,a)+\chi\qquad\mbox{weakly two-scale in $L^2$.}
\end{equation}
(By the argument in \cite[Proof of Theorem~3.5 (a), Step~4]{N12} it a posteriori follows that the rotation field $R$ in \eqref{eq:convergence1} and \eqref{lim:compact2b} are the same). Hence, it is left to show \eqref{eq:convergence2}. By two-scale convergence, for every $\eta\in L^2(\omega)$ we have
\begin{align}\label{lim:t:compact1}
 \int_\omega \left(\fint_S \Eop_h(u^h)\cdot (e_1\otimes e_1)\,d\bar x\right)\eta(x_1)\,dx_1 \to \frac1{|S|}\int_{\Omega\times Y} \left( \Eop (R^t\partial_1R,a)+\chi\right)\cdot (e_1 \otimes e_1)\eta\,dx\,dy.
\end{align}
By \eqref{ass:S} and the definition of $\H_{\rm rel}^\gamma$ , we have for almost every $x_1\in\omega$,
\begin{align}\label{lim:t:compact2}
 \int_{S\times Y}\chi\cdot (e_1\otimes e_1)\,d\bar x\,dy=0,\qquad \int_{S\times Y}  \Eop (R^t\partial_1R(x_1),a(x_1))\cdot(e_1\otimes e_1)\,d\bar x\,dy= |S|a(x_1).
\end{align}
Combining \eqref{lim:t:compact1} and \eqref{lim:t:compact2}, we obtain \eqref{eq:convergence2}.

\step{2} Proof of (b).

Let $(u,R,a)\in\mathcal A$ be given. By part (b) of \cite[Theorem~3.5]{N12}, we find $(u^h)\subset L^2(\Omega;\R^3)$ satisfying \eqref{eq:convergence1} and \eqref{lim:compact2}. The same argument as in Step~1 yields that $(u^h)$ also satisfies \eqref{eq:convergence2}. 

\end{proof}

We recall the following (lower semi-)continuity  result with respect to two-scale convergence:
\begin{lemma}[\cite{N12}, Lemma~4.8] \label{L:2sconvex}
 Fix $\gamma\in[0,\infty]$. Let $Q_\e$ and $Q$ be as in Assumption~\ref{ass:W}, and let $\tilde E_h$ be a sequence in $L^2(\Omega;\R^{3\times 3})$.
 \begin{itemize}
  \item[(a)] If $\tilde E_h\stackrel{2,\gamma}{\wto} \tilde E$ weakly two-scale in $L^2$, then
  \begin{equation*}
   \liminf_{h\to0}\int_\Omega Q_{\e(h)}(x,\tilde E_h(x))\,dx\geq  \iint_{\Omega\times Y}Q(x,y,\tilde E(x,y))\,dy\,dx.
  \end{equation*}
  \item[(b)] If $\tilde E_h\stackrel{2,\gamma}{\to}\tilde E$ strongly two-scale in $L^2$, then
  \begin{equation*}
   \lim_{h\to0}\int_\Omega Q_{\e(h)}(x,\tilde E_h(x))\,dx=\iint_{\Omega\times Y}Q(x,y,\tilde E(x,y))\,dy\,dx.
  \end{equation*}
 \end{itemize}

\end{lemma}

Notice that in \cite{N12} the statement of Lemma~\ref{L:2sconvex} is proven, following arguments of \cite{Vis07}, under the assumption that $x_1\mapsto Q(x_1,\bar x,y)$ is continuous for almost every $(\bar x,y)\in S\times \R$. Evidently, this extends to the piecewise continuous case considered here.

\smallskip

Now, we are in position to prove Theorem~\ref{Th:gamma0}. We follow the argument in \cite{N12}.

\begin{proof}[Proof of Theorem~\ref{Th:gamma0}]

\step{1} Compactness.
 
In view of Proposition~\ref{P:comact2} it suffices to show that for every sequence $(u^h)\subset L^2(\Omega;\R^3)$
\begin{equation*}%\label{enebdbending}
 \limsup_{h\to0}\mathcal I^{\e(h),h}(u^h)<\infty\qquad\mbox{implies}\qquad  \limsup_{h\to0}\frac1{h^2}\int_\Omega \dist^2(\nabla_h u^h,\SO 3)\,dx<\infty.
\end{equation*}
By the triangle inequality, 
\begin{align*}
 \dist(\nabla_h u^h,\SO 3)\leq& \dist(\nabla_h u^h(\Id+h B^h),\SO 3)+h|\nabla_h u^hB^h| \\
 \leq& \dist(\nabla_h u^h(\Id+h B^h),\SO 3)+h|B^h|\big(\dist(\nabla_h u^h,\SO 3)+|\Id|\big).
\end{align*}
Since $\limsup_{h\to0} h\|B^h\|_{L^\infty(\Omega)}=0$ (by assumption), there exists $h_0>0$ such that $h\|B^h\|_{L^\infty(\Omega)}<\frac12$ for $h\in(0,h_0]$. Hence, for $h\in(0,h_0]$,
\begin{equation*}
 \frac1{2h}\|\dist(\nabla_h u^h,\SO 3)\|_{L^2(\Omega)}\leq \frac1h\|\dist(\nabla_h u^h(\Id + h B^h),\SO 3)\|_{L^2(\Omega)}+\sqrt 3\|B^h\|_{L^2(\Omega)}.
\end{equation*}
Now, \eqref{est:finitebending} follows by non-degeneracy of $W$, cf.~(W2), and the equiboundedness of $B^h$ in $L^2(\Omega)$.

\step{2} Lower bound.

Let $(u^h)\subset L^2(\Omega;\R^3)$ be such that \eqref{eq:convergence1} and \eqref{eq:convergence2} are valid for some $(u,R,a)\in \mathcal A$. Without loss of generality, we may assume that
\begin{equation*}
 \liminf_{h\to0}\mathcal I^{\e(h),h}(u^h)=\limsup_{h\to0}\mathcal I^{\e(h),h}(u^h)<\infty.
\end{equation*}
By Proposition~\ref{P:comact2} (a), there exists $\chi\in \H_{\rm rel}^\gamma$ such that \eqref{lim:compact2} holds (up to possibly extracting a further subsequence). To shorten notation, we set
\begin{equation*}
  E^h:=\Eop_h(u^h)\qquad\text{and}\qquad E:=\Eop(R^t\partial_1R,a).
\end{equation*}

\substep{2.1} We claim that
\begin{equation}\label{eq:claimliminf1}
 \liminf_{h\to0}\mathcal I^{\e(h),h}(u^h)\geq \iint_{\Omega\times Y}Q(x,y,E+\chi+\sym B)\,dy\,dx.
\end{equation}
The bound follows by a careful Taylor expansion. To that end, set 
\begin{equation*}
 {\bf 1}^h(x):=\begin{cases}
             1&\mbox{if $\frac{\dist(\nabla_h u^h(x),\SO 3)}{h}+|\operatorname E^h(x)|+|B^h(x)|\leq h^{-\frac12}$}\\
             0&\mbox{else.}
            \end{cases}
\end{equation*}
By construction there exists $h_0>0$ such that for all $h\in(0,h_0]$ we have $\det (\nabla_h u(x))>\frac12$ on $\{{\bf 1}^h=1\}$. Thus, by polar-factorization, for all $h\in(0,h_0]$ and $x\in\{{\bf 1}^h=1\}$ there exists $R^h(x)\in\SO 3$ such that $
\nabla_h u(x)=R^h(x)(\Id+h \operatorname{E}^h(x))$.
Thanks to the non-negativity of $W_{\e(h)}$, frame-indifference (W1), the quadratic expansion (W4), and minimality at identity (W3), we get
\begin{eqnarray*}
  &&\frac1{h^2}W_{\e(h)}(x,\nabla_h u^h(\Id + h B^h))\geq\frac{{\bf 1}^h(x)}{h^2}W_{\e(h)}(x,\nabla_h u^h(\Id + h B^h))\\
  &=&\frac{1}{h^2}W_{\e(h)}\big(x,\Id + {\bf 1}^h(x)h(E^h+B^h+h E^hB^h )\big)\\
  &\geq& Q_{\e(h)}(x,{\bf 1}^h(x)(E^h+B^h+hE^hB^h))- |E^h(\Id+hB^h)+B^h|^2r(3\sqrt{h}),
\end{eqnarray*}
where in the last line we used that $h|{\bf 1}^h(E^h+B^h+h E^h B^h)|\leq 3\sqrt{h}$. Since $(E^h)$ and $(B^h)$ are bounded in $L^2$ and since $\limsup_{h\to0}h\|B^h\|_{L^\infty(\Omega)}=0$ (by assumption), we get
\begin{equation*}
 \liminf_{h\to 0} \mathcal I^{\e(h),h}(u^h)\geq \liminf_{h\to0}\int_\Omega Q_{\e(h)}(x,{\bf 1}^h(x)( E^h+B^h+h E^h B^h ))\,dx.
\end{equation*}
From $E^h\stackrel{2,\gamma}{\to}E$, $B^h\stackrel{2,\gamma}{\to}B$ and $\limsup_{h\to0}h\|B^h\|_{L^\infty(\Omega)}=0$, and since ${\bf 1}^h\to 1$ (boundedly in measure), we get
$$
 {\bf 1}^h(E^h+B^h+h E^h B^h)\stackrel{2,\gamma}{\rightharpoonup} \operatorname E + B,
$$
and thus  \eqref{eq:claimliminf1} follows with help of Lemma~\ref{L:2sconvex}.

\substep{2.2} Conclusion

Recall that $P^{\gamma,x_1}E(x_1,\cdot)=0$ for a.e.\ $x_1\in\omega$. Hence,
\begin{align*}
  &\iint_{\Omega\times Y}Q(x,y,E+\sym B+\chi)\,dy\,dx\geq \int_\omega \inf_{\chi\in \H_{\rm rel}^\gamma}\|E+\sym B+\chi\|_{x_1}^2\,dx_1\\
  \stackrel{\eqref{eq:pythagoras}, \eqref{def:m}}{=}&\int_{\omega}\left\|P_{\rm rel}^{\gamma,x_1}(E+\sym B)\right\|_{x_1}^2\,dx_1+m.
\end{align*}
By Definition~\ref{D:eff} we have $P_{\rm rel}^{\gamma,x_1}\sym B(x_1,\cdot)=P_{\rm rel}^{\gamma,x_1}\Eop(K_{\rm eff}(x_1),a_{\rm eff}(x_1))$ for a.e.\ $x_1\in \omega$, and thus with \eqref{eq:claimliminf1} we get
\begin{align*}
 \liminf_{h\to0}\mathcal I^{\e(h),h}(u^h)\geq& \int_{\omega}\left\|P_{\rm rel}^{\gamma,x_1}\Eop(R^t\partial_1R+K_{\rm eff},a+a_{\rm eff})\right\|_{x_1}^2\,dx_1+m\\
 \stackrel{\eqref{eq:qhomproj}}{=}&\int_\omega Q_{\rm hom}(x_1,R^t\partial_1R+K_{\rm eff},a+a_{\rm eff})\,dx_1+m.
\end{align*}

\step 3 Limsup inequality.

Let $(u,R,a)\in\mathcal A$ be given. For convenience set $E:=\Eop(R^t\partial_1R,a)$. By Lemma~\ref{L:pythagoras}, \eqref{eq:qhomproj}, and Definition~\ref{D:eff} we have 
\begin{eqnarray}\notag
  &&\int_\omega Q_{\rm hom}(x_1,K+K_{\rm eff},a+a_{\rm eff})\,dx_1+m\\\notag
  &\stackrel{\eqref{eq:qhomproj},\eqref{def:m}}{=}&\int_\omega\|P^{\gamma,x_1}_{\rm rel}(\Eop(K+K_{\rm eff},a+a_{\rm eff}))\|^2_{x_1}+\|P^{\gamma,x_1}\sym B\|^2_{x_1}\,dx_1\\\label{eq:sss1}
  &\stackrel{\eqref{def:Keff}}{=}&\int_\omega\|P^{\gamma,x_1}_{\rm rel}(E+\sym B)\|^2_{x_1}+\|P^{\gamma,x_1}\sym B\|^2_{x_1}\,dx_1.
\end{eqnarray}
In view of \eqref{decomposition} we have $I=P^{\gamma,\bullet}+P^{\gamma,\bullet}_{\rm rel}+(I-P^{\gamma,\bullet}_{\rm rel})(I-P^{\gamma,\bullet})$, where $I$ denotes the identity operator on $L^2(\omega;\H)$. Since $P^{\gamma,\bullet}E=0$, we get the decomposition
\begin{eqnarray*}
  E+\sym B&=&P^{\gamma,\bullet}\sym B+P^{\gamma,\bullet}_{\rm rel}(E+\sym B)+\underbrace{(1-P^{\gamma,\bullet}_{\rm rel})(1-P^{\gamma,\bullet})(E+\sym B)}_{=:-\chi}.
\end{eqnarray*}
Combined with \eqref{eq:sss1} we arrive at
\begin{equation}\label{eq:sss2}
  \int_\omega Q_{\rm hom}(x_1,R^t\partial_1R+K_{\rm eff},a+a_{\rm eff})\,dx_1+m =\iint_{\Omega\times Y} Q(x,y,E+\sym B+\chi)\,dydx.
\end{equation}
By construction we have $\chi\in L^2(\omega;\H^\gamma_{\rm rel})$, and thus by Proposition~\ref{P:comact2} there exists a sequence $u^h\in H^1(\Omega;\R^3)$ satisfying \eqref{eq:convergence1}, \eqref{eq:convergence2}, and
\begin{align}
 &E^h:=\Eop_h(u^h)\stackrel{2,\gamma}{\to} E + \chi\quad\mbox{strongly two-scale in $L^2$,}\label{limsup:2}\\
 &\limsup_{h\to0}\esssup_{x\in\Omega}\sqrt h\left(\frac{\dist(\nabla_h u^h(x),\SO 3)}{h}+|E^h|\right)=0.\label{limsup:3}
\end{align}
\eqref{limsup:3} implies $\det \nabla_h u^h>0$ for all $h>0$ sufficiently small. Hence, combining the polar-factorization of $\nabla_h u^h$, frame indifference of $W_{\e(h)}$, and (W4), we get
\begin{align*}
 \frac1{h^2}W_{\e(h)}(\cdot,\nabla_h u^h)=&\frac1{h^2}W(\cdot,\Id+h(E^h+B^h+hE^hB^h))\\
 =&Q_{\e(h)}(\cdot,E^h+B^h+hE^hB^h)+q^h,\\
 \mbox{where}\qquad& |q^h|\leq \left|E^h+B^h+hE^h B^h\right|^2r\left(h|E^h+B^h+hE^hB^h|\right)
\end{align*}
Since $(E^h)$ and $(B^h)$ are uniformly bounded in $L^2(\Omega)$, and $\limsup_{h\to0}h(\|B^h\|_{L^\infty(\Omega)}+\|E^h\|_{L^\infty(\Omega)})=0$ by~\eqref{ass:eqB} and \eqref{limsup:3}, we obtain that $q^h\to0$ in $L^1(\Omega)$. Moreover, \eqref{limsup:2}, Assumption~\ref{ass:B} and \eqref{ass:eqB} imply
\begin{align*}
E^h+B^h+hE^hB^h\stackrel{2,\gamma}{\longrightarrow} & E + \chi+ B.
\end{align*}
Hence, Lemma~\ref{L:2sconvex} and \eqref{eq:sss2} yield
\begin{eqnarray*}
\limsup_{h\to0}\mathcal I^{\e(h),h}(u^h)&=&\iint_{\Omega\times Y}Q(x,y,E+\sym B+\chi)\,dy\,dx\\
&=&  \int_\omega Q_{\rm hom}(x_1,R^t\partial_1R+K_{\rm eff},a+a_{\rm eff})\,dx_1+m.
\end{eqnarray*}
\end{proof}

\subsection{Homogenization formulas via BVPs -- Proofs of Lemma~\ref{L:represent}, Proposition~\ref{P:BVP} and Lemma~\ref{L:BVP}}\label{sec:proofbvp}
\begin{proof}[Proof of Lemma~\ref{L:represent}]
Throughout the proof we denote by $C$ a positive constant that can be chosen only depending on $\alpha,\beta,\gamma$ and $S$.
By construction ${\Eop}^{\gamma,x_1}$ is linear and bounded. A standard argument from functional analysis implies that ${\Eop}^{\gamma,x_1}$ is an isomorphism, if  ${\Eop}^{\gamma,x_1}$ is surjective and satisfies \eqref{estimate}. Surjectivitiy follows from the fact that $\H^\gamma=\mbox{range}(\Eop^{\gamma,x_1})\oplus\H^\gamma_{\rm rel}$, which implies that $(\H^\gamma_{\rm rel})^{\perp x_1}\subset \mbox{range}({\Eop}^{\gamma,x_1})$ for every $x_1\in \omega$. The upper bound in \eqref{estimate} is a consequence of \eqref{equivalence}. We prove the lower bound. Since $(\H^{\gamma}_{\rm rel})^{\perp x_1}\subset \H^\gamma$, for any $(K,a)\in\Skew 3\times\R$ there exists $\chi_{K,a}\in \H^\gamma_{\rm rel}$ such that ${\Eop}^{\gamma,x_1}(K,a)={\Eop}(K,a)+\chi_{K,a}$. By the properties of the orthogonal projection, 
\begin{equation*}
  \forall \chi\in\H^\gamma_{\rm rel}\,:\qquad \big({\Eop}(K,a)+\chi_{K,a},\chi\big)_{x_1}=0.
\end{equation*}
This variational problem has a unique solution, and the associated solution operator $T:\Skew 3\times\R\to\H^\gamma_{\rm rel}$, $T(K,a):=\chi_{K,a}$ is linear and bounded. By the lower bound in \eqref{equivalence}, we thus obtain
\begin{eqnarray*}
  \|{\Eop}^{\gamma,x_1}(K,a)\|_{x_1}^2 &=&  \|{\Eop}(K,a)+T(K,a)\|_{x_1}^2\geq\frac{1}{C}\|{\Eop}(K,a)+T(K,a)\|^2,
\end{eqnarray*}
where $\|\cdot\|$ and $\big(\cdot,\cdot\big)$ denote the standard norm and scalar product in $\H$ (i.e.~$(F,G):=\sum_{ij}\iint_{S\times Y}F_{ij}G_{ij}$). It is easy to see that $E(0,a)$ is $(\cdot,\cdot)$-orthogonal to $E(K,0)$ and $\H^\gamma_{\rm rel}$. Thus,
\begin{eqnarray*}
 \|{\Eop}(K,a)+T(K,a)\|^2&=& \|{\Eop}(K,0)+T(K,a)\|^2+ \|{\Eop}(0,a)\|^2\geq \inf_{\chi\in\H^\gamma_{\rm rel}}\|{\Eop}(K,0)+\chi\|^2+ \|{\Eop}(0,a)\|^2\\
 &\geq&\inf_{\chi\in\H^\gamma_{\rm rel}}\|{\Eop}(K,0)+\chi\|^2+ \frac{1}{C}a^2.
\end{eqnarray*}
Moreover, by the short argument in \cite[Step 3, Proof of Proposition~2.13]{N12}, we have
\begin{equation*}
  \inf_{\chi\in\H^\gamma_{\rm rel}}\|{\Eop}(K,0)+\chi\|^2\geq\frac{1}{C}|K|^2.
\end{equation*}
This completes the argument for the lower bound in \eqref{equivalence}.
\end{proof}

\begin{proof}[Proof of Proposition~\ref{P:BVP}]

\step{1} Argument for (c).

The definition of $P^{\gamma,x_1}_{\rm rel}$ and \eqref{eq:cell1} yield for $i=1,\ldots,4$ the identity
\begin{equation}\label{eq:sss3}
  P_{\rm rel}^{\gamma,x_1}E^{(i)}=E^{(i)}+\chi^{(i)}(x_1).
\end{equation}
Hence, by \eqref{eq:qhomproj}
\begin{align*}
 Q_\ho(x_1,K,a)=&\|P_{\rm rel}^{\gamma,x_1}\Eop (K,a)\|_{x_1}^2=\|\sum_{i=1}^4k_i(x_1) P_{\rm rel}^{\gamma,x_1} E^{(i)}\|_{x_1}^2\stackrel{\eqref{def:M}}{=}\sum_{i,j=1}^4k_i(x_1)k_j(x_1) \mathbb M_{ij}(x_1),\end{align*}
which proves the claim.

\step 2 Argument for (a).

The symmetry of $\mathbb M$ is obvious (see \eqref{def:M}). Since $Q(x_1,\bar x,y,\cdot)\in\mathcal Q(\alpha,\beta)$ for almost every $(\bar x,y)\in S\times \R$ (see Assumption~\ref{ass:W}), the identities \eqref{eq:qhomproj} and \eqref{eq:qhomkmk} yield 
\begin{align*}
 \alpha\|P_{\rm rel}^{\gamma,x_1}\Eop (K,a)\|_{L^2(S\times Y)}^2\leq k\cdot \mathbb M(x_1)k\leq \beta \|P_{\rm rel}^{\gamma,x_1}\Eop (K,a)\|_{L^2(S\times Y)}^2
\end{align*}
for every $(K,a)\in\Skew 3\times \R$. Hence, \eqref{estimate} implies that $\mathbb M$ is positive definite.

\step 3 Argument for (b).

By the definition of $\mathbb M$ it suffices to show that the map $x_1\mapsto \chi^{(i)}(x_1)\in L^2(S\times Y)$ is piecewise continuous. For every $x_1,x_1'\in \omega$ and $i=1,\dots,4$, we have 
\begin{align*}
 &\alpha \|\chi^{(i)}(x_1)-\chi^{(i)}(x_1')\|_{L^2(S\times Y)}^2\\
 \leq& \|\chi^{(i)}(x_1)-\chi^{(i)}(x_1')\|_{x_1}^2\\=&\big(E^{(i)}+\chi^{(i)}(x_1'),\chi^{(i)}(x_1)-\chi^{(i)}(x_1')\big)_{x_1'}-\big(E^{(i)}+\chi^{(i)}(x_1'),\chi^{(i)}(x_1)-\chi^{(i)}(x_1')\big)_{x_1}.
\end{align*}
Hence, the piecewise continuity of $\mathbb L(x_1,\cdot)\in L^\infty(S\times \R)$ yields the piecewise continuity of $x_1\mapsto \chi^{(i)}(x_1)\in L^2(S\times Y)$.

\step 4 Argument for (d).

Let $k:\omega\to\R^4$ be defined via the identity \eqref{def:kinpbvp}. For every $j\in\{1,\dots,4\}$ and almost every $x_1\in\omega$, we have
\begin{eqnarray*}
 e_j\cdot \mathbb M(x_1)k(x_1)&=&\sum_{i=1}^4k_i(x_1) \left(E^{(j)}+\chi^{(j)}(x_1),E^{(i)}+\chi^{(i)}(x_1)\right)_{x_1}\\
 &\stackrel{\eqref{eq:sss3}}{=}& \big(E^{(j)}+\chi^{(j)}(x_1),P_{\rm rel}^{\gamma,x_1}\Eop (K_{\rm eff},a_{\rm eff})\big)_{x_1}\\
&\stackrel{\eqref{def:Keff}}{=}&\big( E^{(j)}+\chi^{(j)}(x_1),P_{\rm rel}^{\gamma,x_1}\sym B(x_1,\cdot)\big)_{x_1}\\
&\stackrel{\eqref{eq:cell1}}{=}&\big( E^{(j)}+\chi^{(j)}(x_1),B(x_1,\cdot)\big)_{x_1}= e_j\cdot b(x_1),
\end{eqnarray*}
which proves the claim.
 
\end{proof}

\begin{proof}[Proof of Lemma~\ref{L:BVP}]
\step 1 The case $\gamma\in(0,\infty)$. It is easy to see that the map
\begin{equation*}
 \iota^\gamma:\big\{\phi\in H^1(S\times\mathcal Y;\R^3)\text{ satisfies }\eqref{eq:sidecondition01}\,\big\}\to \H^\gamma_{\rm rel},\qquad\phi\mapsto \sym[D^\gamma\nabla\phi]
\end{equation*}
defines a linear and bounded surjection. Thanks to Korn's inequality in form of \cite[Proposition~6.12]{N12}, $\iota^\gamma$ is also injective, and thus an isomorphism. Thus,  $\phi_E:=({\iota^\gamma})^{-1}\chi_E$ is characterized by the equation
\begin{equation*}
 (E+\sym[D^\gamma\nabla\phi_E],\sym[\iota^\gamma\phi])_{x_1}=0\qquad\text{for all }\phi\in H^1(S\times\mathcal Y;\R^3)\text{ subject to }\eqref{eq:sidecondition01}.
\end{equation*}
By the definition of $(\cdot,\cdot)_{x_1}$, and thanks to the fact that $\langle\mathbb L F,G\rangle=\langle\mathbb L \sym F,\sym G\rangle$, the above equation can be written in the form
\begin{equation*}
 \iint_{S\times Y}\langle\mathbb L(x_1,\bar x,y)(E+D^\gamma\nabla\phi_E),D^\gamma\nabla\phi\rangle=0.
\end{equation*}

\step 2 The case $\gamma=0$. Similarly to the previous step, we only need to show that $\iota^0$ is an isomorphism. By construction $\iota^0$ is linear, bounded and by the following observation surjective:
\begin{equation*}
\forall \hat\phi\in H^1(\mathcal Y;\R^2):\quad\sym \left[\begin{pmatrix}
             0 & -(\bar \nabla \bar \phi)^t \\ \partial_y \hat \phi&0
            \end{pmatrix}\right]=0\quad \mbox{with $\bar \phi:= \partial_y \hat\phi \cdot \bar x \in L^2(\mathcal Y;H^1(S))$.}
\end{equation*}
We argue that the kernel of $\iota^0$, denoted by $\mbox{kern}(\iota^0)$, is trivial. By orthogonality we have
\begin{equation*}
 \|\iota^0(\Psi,\hat\phi,\bar\phi)\|^2=    \|\iota^0(\Psi_\kappa,0,0)\|^2+ \|\iota^0(\Psi_\tau,0,\bar\phi)\|^2+\|\iota^0(0,\hat\phi,0)\|^2,
\end{equation*}
where $\Psi_\tau=(\Psi\cdot K^{(4)}) K^{(4)}$, and $\Psi_\kappa:=\Psi-\Psi_\tau$. Hence, (thanks to Poincare's inequality)
\begin{equation*}
 (\Psi,\hat\phi,\bar\phi)\in\mbox{kern}(\iota^0)\qquad\Rightarrow\qquad\Psi_\kappa,\,\hat\phi\text{ are $0$, and }\|\iota^0(\Psi_\tau,0,\bar\phi)\|=0.
\end{equation*}
Let $\bar\phi_i$ denote the components of $\bar\phi$. By orthogonality, we have
\begin{equation*}
 \|\iota^0(\Psi_\tau,0,\bar\phi)\|^2=\|\iota^0(\Psi_\tau,0,\bar\phi_1e_1)\|^2+\|\iota^0(0,0,\bar\phi_2e_2)\|^2+\|\iota^0(0,0,\bar\phi_3e_3)\|^2.
\end{equation*}
Thus, $\bar \phi_2=\bar \phi_3=0$. To see that $\Psi_\tau$ and $\bar \phi_1$ vanish as well, note that $\|\iota^0(\Psi_\tau,0,\bar\phi_1e_1)\|^2=0$ is equivalent to
\begin{equation*}
 \iint_{S\times Y}|\tau(y) K^{(3)}\bfx(\bar x)-\bar\nabla\bar \phi_1(\bar x,y)|^2=0,\qquad \tau(y):=\partial_y\Psi\cdot K^{(3)}.
\end{equation*}
Since $K^{(4)}\bfx(\bar x)$ is not a gradient in $\bar x$, we deduce that $\phi_1=0$ and $\tau(y)=0$, i.e. $\Psi_\tau=0$.

\step 3 The case $\gamma=\infty$. As in the previous step, it suffices to argue that the kernel of $\iota^\infty$ is trivial. By orthogonality,
\begin{equation*}
 \|\iota^\infty(\hat\phi,\bar\phi)\|^2=    \|\iota^\infty(\hat\phi,0)\|^2+    \|\iota^\infty(0,\bar\phi)\|^2,
\end{equation*}
and thus $\iota^\infty(\hat\phi,\bar\phi)=0$ implies (thanks to Poincar\'e's inequality) $\hat\phi=0$ and $\bar\phi=0$.

\end{proof}

\subsection{Isotropic case -- Proof of Lemma~\ref{L:isotropic}}\label{sec:proofisotropic}

\begin{proof}[Proof of Lemma~\ref{L:isotropic}]

We only need to prove \eqref{corr4}--\eqref{corrtorsion} (which will done in Step~1 and Step~2 below). The remaining claims then follow from the observation that $ {\Eop}(K^{(j)},0)+\chi^{(j)}$, $j=1,2,3$ and ${\Eop}(0,1)+\chi^{(4)}$ given by \eqref{corr4}--\eqref{corrtorsion} are mutually orthogonal with respect to the inner product $(\cdot,\cdot)_{x_1}$. The latter implies that $\mathbb M(x_1)$ is diagonal and thus a straightforward calculation yields the precise formulas for the entries of $\mathbb M$, $b$ and $k$. For the argument of  \eqref{corr4}--\eqref{corrtorsion} we first note that $\chi^{(i)}$ defined by the variational problem \eqref{eq:cell1} can be equivalently characterized as the minimizer of an associated quadratic--convex energy functional. We exploit this fact in the proof below. 

\step 1 The case $\gamma=0$.
\begin{itemize}
\item Argument for \eqref{corr4}. In view of  Lemma~\ref{L:BVP}, we have $\chi^{(1)}=\iota^0(\Psi,\hat\phi,\bar\phi)$, where $(\Psi,\hat\phi,\bar\phi)$ is a minimizer of the functional ${\bf X^0}\ni (\Psi,\hat\phi,\bar\phi)\mapsto \int_{S\times Y}Q(y,e_1\otimes e_1+\iota^0(\Psi,\hat\phi,\bar\phi))$. Set
\begin{equation*}
 \partial_y\Psi=\kappa_2 K^{(2)}+\kappa_3 K^{(3)}+\tau K^{(4)},\qquad \bar\phi=(\bar\phi_1,\bar{\bar\phi}),\qquad \bar{\bar\phi}=(\bar\phi_2,\bar\phi_3),
\end{equation*}
where $K^{(i)}$, $i=1,2,3$ is the basis of $\Skew 3$ given in \eqref{def:K}. Note that
\begin{align*}
 &\sym[e_1\otimes e_1+\iota^0(\Psi,\hat\phi,\bar\phi)]\\
 &=\left(\begin{array}{c|c}
          1+\partial_y\hat\phi_1-\kappa\cdot\bar x & \frac12\big(\partial_2\bar\phi_1-\tau x_3\big)\qquad \frac12\big(\partial_3\bar\phi_1+\tau x_2\big)\\\hline
        \begin{array}{c}
          *\\
          *
        \end{array}&\sym[\bar\nabla\bar{\bar\phi}]
 \end{array}\right),
\end{align*}
where $\kappa=(\kappa_2,\kappa_3)$. By appealing to the specific structure of $Q$ we directly see that the minimizer satisfies $\tau=\bar\phi_1=0$, and that the remaining degrees of freedom $\kappa$, $\hat\phi_1$, and $\bar{\bar\phi}:=(\bar\phi_2,\bar\phi_3)$ minimize the expression
\begin{align*}
\iint_{S\times Y} 2\mu\big(1+\partial_y\hat\phi_1-\kappa\cdot\bar x\big)^2+2\mu|\sym\bar\nabla\bar{\bar\phi}|^2+\lambda(1+\partial_y\hat\phi_1-\kappa\cdot\bar x+\bar\nabla\cdot\bar{\bar\phi})^2.
\end{align*}
The Euler-Lagrange equation reads
\begin{eqnarray*}
 0&=&\iint_{S\times Y} 2\mu\big(1+\partial_y\hat\phi_1-\kappa\cdot\bar x\big)\big(\partial_y\delta\hat\phi-\delta k\cdot\bar x\big)+2\mu\sym\bar\nabla\bar{\bar\phi}\cdot\bar\nabla\delta\bar{\bar\phi} \\
 &&\qquad+\lambda\big(1+\partial_y\hat\phi_1-\kappa\cdot\bar x+\bar\nabla\cdot\bar{\bar\phi}\big)\big(\partial_y\delta\hat\phi-\delta\kappa\cdot\bar x+\bar\nabla\cdot\delta\bar{\bar\phi}\big)
\end{eqnarray*}
for all $\delta\hat\phi\in H^1(\mathcal Y)$, $\delta\kappa=(\delta\kappa_2,\delta\kappa_3)\in L^2(\mathcal Y;\R^2)$ and $\delta \bar{\bar\phi}\in L^2(S;H^1(\mathcal Y;\R^2))$. With the Ansatz $\kappa=0$, $\bar{\bar\phi}=\rho\bar x$ with $\rho\in L^2(\mathcal Y)$, we have $\bar\nabla\cdot\bar{\bar\phi}=2\rho$ and $\sym\bar\nabla\bar{\bar\phi}=\rho \Id$, and the above turns into
\begin{eqnarray*}
 0&=&\iint_{S\times Y} \big((2\mu+\lambda)(1+\partial_y\hat\phi_1)+2\lambda\rho\big)\partial_y\delta\hat\phi +\big(2\rho(\lambda+\mu)+\lambda(1+\partial_y\hat\phi)\big)\bar\nabla\cdot\delta\bar{\bar\phi},
\end{eqnarray*}
where we used that $\delta \kappa \cdot\bar x$ is orthogonal to $L^2(\mathcal Y)$. From the variation in $\delta\bar{\bar\phi}$, we get $\rho=-\nu(1+\partial_y\hat\phi)$, which we plug into the Euler-Lagrange equation:
\begin{eqnarray*}
 0&=&\int_Y \big((2\mu+\lambda)(1+\partial_y\hat\phi_1)-\frac{\lambda^2}{(\lambda+\mu)}(1+\partial_y\hat\phi_1)\big)\partial_y\delta\hat\phi =\int_Y \beta(1+\partial_y\hat\phi)\partial_y\delta\hat\phi
\end{eqnarray*}
Thus, $(1+\partial_y\hat\phi)=\frac{\beta_{\hom}}{\beta}$, and the claim follows.
\item Argument for \eqref{corr1}. We only consider $i=2$, the case $i=3$ follows in the same way. As above, we have $\chi^{(2)}=\iota^0(\Psi,\hat\phi,\bar\phi)$, where $(\Psi,\hat\phi,\bar\phi)$ is a minimizer of the functional ${\bf X^0}\ni (\Psi,\hat\phi,\bar\phi)\mapsto \int_{S\times Y}Q(y, -x_2 e_1\otimes e_1+\iota^0(\Psi,\hat\phi,\bar\phi))$. Analogous considerations as in the argument for \eqref{corr4} yield $\tau=\bar\phi_1=0$, and that the remaining degrees of freedom $\kappa$, $\hat\phi_1$, and $\bar{\bar\phi}:=(\bar\phi_2,\bar\phi_3)$ minimize the expression
    \begin{align*}
      \iint_{S\times Y} 2\mu\big(\partial_y\hat\phi_1-x_2-\kappa\cdot\bar x\big)^2+2\mu|\sym\bar\nabla\bar{\bar\phi}|^2+\lambda(\partial_y\hat\phi_1-x_2-\kappa\cdot\bar x+\bar\nabla\cdot\bar{\bar\phi})^2.
    \end{align*}
The Euler-Lagrange equation reads
\begin{eqnarray*}
      0&=&\iint_{S\times Y} 2\mu\big(\partial_y\hat\phi_1-x_2-\kappa\cdot\bar x\big)\big(\partial_1\delta\hat\phi-\delta k\cdot\bar x\big)+2\mu\sym\bar\nabla\bar{\bar\phi}\cdot\bar\nabla\delta\bar{\bar\phi} \\
      &&\qquad+\lambda\big(\partial_y\hat\phi_1-x_2-\kappa\cdot\bar x+\bar\nabla\cdot\bar{\bar\phi}\big)\big(\partial_y\delta\hat\phi_1-\delta\kappa\cdot\bar x+\bar\nabla\cdot\delta\bar{\bar\phi}\big)
\end{eqnarray*}
for all $\delta\hat\phi\in H^1(\mathcal Y)$, $\delta\kappa=(\delta\kappa_2,\delta\kappa_3)\in L^2(\mathcal Y;\R^2)$ and $\delta \bar{\bar\phi}\in L^2(S;H^1(\mathcal Y;\R^2))$. With the Ansatz $\hat \phi_1=0$, $\kappa_3=0$, $\bar{\bar\phi}=\rho\left(\tfrac12(x_2^2-x_3^2)e_2+x_2x_3 e_3\right)$ with $\rho\in L^2(\mathcal Y)$, we have $\bar\nabla\cdot\bar{\bar\phi}=2\rho x_2$ and $\sym\bar\nabla\bar{\bar\phi}=\rho x_2 \Id$ and the above turns into
\begin{eqnarray*}
 0&=&\iint_{S\times Y} x_2\big((2\mu+\lambda)(1+\kappa_2)-2\lambda\rho \big)\delta\kappa\cdot\bar x +x_2\big(2\rho (\lambda+\mu)-\lambda (1+\kappa_2)\big)\bar\nabla\cdot\delta\bar{\bar\phi},
\end{eqnarray*}
where we used that $x_2$ is orthogonal to $L^2(\mathcal Y)$. Similar as in the argument for \eqref{corr4}, we get $\rho=\nu(1+\kappa_2)$ and $(1+\kappa_2)=\frac{\beta_{\hom}}{\beta}$ and thus the claim follows.

 \item Argument for \eqref{corrtorsion}. We have $\chi^{(4)}=\iota^0(\Psi,\hat\phi,\bar\phi)$, where $(\Psi,\hat\phi,\bar\phi)$ is a minimizer of the functional ${\bf X^0}\ni (\Psi,\hat\phi,\bar\phi)\mapsto \int_{S\times Y}Q(y,K^{(4)}\bfx  +  \iota^0(\Psi,\hat\phi,\bar\phi))$. Note that
    \begin{align*}
      &\sym[K^{(4)}\bfx+\iota^0(\Psi,\hat\phi,\bar\phi)]\\
      &=
      \left(\begin{array}{c|c}
          \partial_y\hat\phi_1-\kappa\cdot\bar x & \frac12\big(\partial_2\bar\phi_1-(\tau+1)x_3\big)\qquad \frac12\big(\partial_3\bar\phi_1+(\tau+1) x_2\big)\\\hline
        \begin{array}{c}
          *\\
          *
        \end{array}&\sym[\bar\nabla\bar{\bar\phi}]
      \end{array}\right).
    \end{align*}
    It is easy to see that the minimizer satisfies $\kappa=\hat\phi=0$ and $\bar{\bar\phi}=0$, and that the remaining degrees of freedom $\tau$ and $\bar \phi$ minimize the expression
    \begin{align*}
      \iint_{S\times Y} \mu \big((\partial_2 \bar\phi - (\tau+1) x_3)^2+(\partial_3 \bar\phi + (\tau+1) x_2)^2\big).
    \end{align*}
    It is straightforward to see that the minimizing pair $\tau$ and $\bar \phi$ are uniquely determined by $1+\tau = \frac{\langle \mu\rangle_\ho}{\mu}$ and $\bar \phi = (1+\tau) \bar \phi_S$ which proves the claim.
  \end{itemize}
  
 \step 2 The case $\gamma=\infty$. 
     \begin{itemize}
  \item Argument for \eqref{corr4}. As in Step~1, we have $\chi^{(1)}=\iota^\infty(\hat\phi,\bar\phi)$, where $(\hat\phi,\bar\phi)$ is a minimizer of the functional ${\bf X^\infty}\ni (\hat\phi,\bar\phi)\mapsto \int_{S\times Y}Q(y,e_1\otimes e_1+\iota^\infty(\hat\phi,\bar\phi))$.  By appealing to the specific structure of $Q$ we directly see that the minimizer satisfies $\hat\phi_2=\hat\phi_3=\bar\phi_1=0$, and that the remaining degrees of freedom $\hat\phi_1$, and $\bar{\bar\phi}:=(\bar\phi_2,\bar\phi_3)$ minimize the expression
  \begin{align*}
      \iint_{S\times Y} 2\mu\big(1+\partial_y\hat\phi_1\big)^2+2\mu|\sym\bar\nabla\bar{\bar\phi}|^2+\lambda(1+\partial_y\hat\phi_1+\bar\nabla\cdot\bar{\bar\phi})^2.
    \end{align*}
    The Euler-Lagrange equation reads
    \begin{equation*}
      0=\iint_{S\times Y} 2\mu\big(1+\partial_y\hat\phi_1\big)\partial_y\delta\hat\phi+2\mu\sym\bar\nabla\bar{\bar\phi}\cdot\bar\nabla\delta\bar{\bar\phi}+\lambda\big(1+\partial_y\hat\phi_1+\bar\nabla\cdot\bar{\bar\phi}\big)\big(\partial_y\delta\hat\phi_1+\bar\nabla\cdot\delta\bar{\bar\phi}\big)
    \end{equation*}
    for all $\delta\hat \phi\in L^2(S;H_1(\mathcal Y))$ and $\delta\bar{\bar\phi}\in H^1(S;\R^2)$. The Ansatz $\bar{\bar\phi}=\rho\bar x$, $\rho\in\R$ yields
    \begin{eqnarray*}
      0&=&\iint_{S\times Y} \big((2\mu+\lambda)(1+\partial_y\hat\phi_1)+2\lambda\rho\big)\partial_y\delta\hat\phi +\big(2\rho(\lambda+\mu)+\lambda(1+\partial_y\hat\phi)\big)\bar \nabla \cdot\delta\bar{\bar\phi}.
    \end{eqnarray*}
    From the variation in $\delta\hat\phi$, we obtain $1+\partial_y\hat\phi_1 =\frac{\langle M\rangle_\ho}{M}+2\rho(\frac{\langle M\rangle_\ho}{M}\langle\frac{\lambda}{M}\rangle-\frac{\lambda}{M})$ (recall $M:=2\mu+\lambda$) and the variation in $\delta \bar{\bar \phi}$ yields 
    \begin{equation*}
     0=2\rho\left(\langle \mu +\lambda\rangle + \langle M\rangle_\ho \langle\frac{\lambda}{M}\rangle - \langle \frac{\lambda^2}{M} \rangle\right) +\langle M\rangle_\ho \langle\frac{\lambda}{M}\rangle \quad \Rightarrow \rho=-\nu_\infty
    \end{equation*}
    and thus the claim follows.
    
      \item Arguments for \eqref{corr1}. As is Step~1, we only consider the case $i=2$. We have $\chi^{(2)}=\iota^\infty(\hat \phi,\bar \phi)$, where $(\hat \phi,\bar \phi)\in X^\infty$ is such that $\hat \phi_2=\hat \phi_3=\bar \phi_1=0$ and the remaining degrees of freedom $\hat \phi_1$ and $\bar{\bar\phi}:=(\bar\phi_2,\bar\phi_3)$ minimize the expression
    \begin{align*}
      \iint_{S\times Y} 2\mu\big(\partial_y\hat\phi_1-x_2\big)^2+2\mu|\sym\bar\nabla\bar{\bar\phi}|^2+\lambda(\partial_y\hat\phi_1-x_2+\bar\nabla\cdot\bar{\bar\phi})^2.
    \end{align*}
    The Euler-Lagrange equation reads
    \begin{equation*}
      0=\iint_{S\times Y} 2\mu\big(\partial_y\hat\phi_1-x_2\big)\partial_y\delta\hat\phi+2\mu\sym\bar\nabla\bar{\bar\phi}\cdot\bar\nabla\delta\bar{\bar\phi}+\lambda\big(\partial_y\hat\phi_1-x_2+\bar\nabla\cdot\bar{\bar\phi}\big)\big(\partial_y\delta\hat\phi+\bar\nabla\cdot\delta\bar{\bar\phi}\big)
    \end{equation*}
    for all $\delta\hat \phi\in L^2(S;H_1(\mathcal Y))$ and $\delta\bar{\bar\phi}\in H^1(S;\R^2)$. We make the Ansatz $\hat\phi_1(\bar x,y) = -x_2 \hat\varphi(y)$ with $\hat \varphi\in H^1(\mathcal Y)$ and $\bar{\bar\phi}=\rho\left(\tfrac12(x_2^2-x_3^2)e_2+x_2x_3 e_3\right)
$ with $\rho\in \R$. Then, $\bar\nabla\cdot\bar{\bar\phi}=2\rho x_2$, $\sym\bar\nabla\bar{\bar\phi}=\rho x_2 \Id$ and the above expression turns into
    \begin{eqnarray*}
      0&=&\iint_{S\times Y} x_2\big((2\mu+\lambda)(\partial_y \hat \varphi+1)-2\lambda\rho \big)\partial_y \delta\hat\phi + x_2\big(2\rho (\lambda+\mu)-\lambda (\partial_y \hat \varphi +1)\big)\bar\nabla\cdot\delta\bar{\bar\phi}.
    \end{eqnarray*}
     From the variation in $\delta\hat\phi$, we obtain $1+\partial_y\hat\varphi_1 =\frac{\langle M\rangle_\ho}{M}-2\rho(\frac{\langle M\rangle_\ho}{M}\langle\frac{\lambda}{M}\rangle-\frac{\lambda}{M})$ and the variation in $\delta \bar \phi$ yields $\rho=\nu_\infty$. Hence, the claim follows.
     
    \item Argument for \eqref{corrtorsion}. We have $\chi^{(4)}=\iota^\infty(\hat\phi,\bar\phi)$, where $(\Psi,\hat\phi,\bar\phi)$ is a minimizer of the functional ${\bf X^0}\ni (\Psi,\hat\phi,\bar\phi)\mapsto \int_{S\times Y}Q(y,K^{(4)}\bfx  +  \iota^\infty(\hat\phi,\bar\phi))$.  Analogous to the case $\gamma=0$, we obtain that $\hat \phi_1=0$ and $\bar{\bar\phi}=0$, and that the remaining degrees of freedom $\hat \phi_2$, $\hat \phi_3$ and $\bar \phi$ minimize the expression
    \begin{align*}
      \iint_{S\times Y} \mu \big((\partial_2 \bar\phi +\partial_y \hat \phi_1- x_3)^2+(\partial_3 \bar\phi + \partial \hat \phi_1 +x_2)^2\big).
    \end{align*}
    Appealing to the corresponding Euler-Lagrange equation, we see that the minimizer is given by $\bar \phi=\bar \phi_S$, $\hat \phi_2=\hat \varphi (\partial_2 \bar\phi-x_3)$, $\hat \phi_2=\hat \varphi (\partial_2 \bar\phi-x_3)$ where $\hat \varphi\in H^1(\mathcal Y)$ is uniquely characterized by $\fint_Y \hat \varphi=0$ and $1+\partial_y \hat\varphi=\frac{\langle \mu\rangle_\ho}{\mu}$. From this, the claim follows.
  \end{itemize}
\end{proof}

\section{Acknowledgments}
The authors are grateful for the support by the DFG in the context of TU Dresden's Institutional Strategy \textit{``The Synergetic University''}. 

\appendix

\section{Appendix}

This appendix contains some supplementary calculations for Section~\ref{par:liquid}. The main problem is to obtain a semi-explicit characterization of the torsion function $\varphi_S$, see \eqref{def:qs}, where $S=(-1,1)^2$ is the unit cube. We provide this characterization in Section~\ref{sec:torsion}.

\subsection{Supplementary calculations for Section~\ref{par:liquid}}\label{appendix:1}

\begin{itemize}
\item Calculations for $k_i$, $i=1,2,3$, i.e.\ the preferred infinitesimal stretch and bending. Set $B(n)=c_1\Id+c_2n\otimes n$ where $n=n(x_3)$ and $|n|=1$ a.e. By Lemma~\ref{L:isotropic}, 
 \begin{align*}
 b_1(n)=&\beta_\gamma\int_S (c_1+c_2 n_1^2(x_3))\,dx_2\,dx_3=4\beta_\gamma c_1+2\beta_\gamma c_2 \int_{-1}^1n_1^2(x_3)\,dx_3\\%\label{lc:b1}\\
 b_2(n)=&-\beta_\gamma\int_S x_2(c_1+c_2 n_1^2(x_3))\,dx_2\,dx_3=0\\%\label{lc:b2}\\
 b_3(n)=&-\beta_\gamma\int_S x_3(c_1+c_2 n_1^2(x_3))\,dx_2\,dx_3=-2c_2 \beta_\gamma\int_{-1}^1 x_3 n_1^2(x_3)\,dx_3.%%\label{lc:b3}
\end{align*}
Using $n_\vartheta^{(S)}(x)\cdot e_1=n_\vartheta^{(T)}(x)\cdot e_1=\cos(\vartheta+\frac\pi4 x_3)$, $c_1=\frac{\bar r}3$ and $c_2=-\bar r$, we easily deduce the claimed values for $k(n_\vartheta^{(S)})\cdot e_i$ and  $k(n_\vartheta^{(T)})\cdot e_i$ with $i=1,2,3$. 
\item Calculations for $k_4$, i.e.\ the preferred twist. Here we have to distinguish between splay bend and twist. 

\underline{'Splay bend':} We claim $b_4(n_\vartheta^{(S)})=0$. Note that $n_\vartheta^{(S)}\cdot e_2\equiv 0$ for all $\vartheta$. Hence, $B(n_\vartheta^{(S)})_{12}=B(n_\vartheta^{(S)})_{21}=0$, $B(n_\vartheta^{(S)})_{13}=B(n_\vartheta^{(S)})_{31}=\bar r(n_\vartheta^{(S)}\cdot e_1) (n_\vartheta^{(S)}\cdot e_3)$ and Lemma~\ref{L:isotropic} yield
\begin{align*}
 b_4(n_\vartheta^{(S)})=&-2\bar r\langle \mu \rangle_\ho \int_S (\partial_3 \varphi_S +x_2)(n_\vartheta^{(S)}(x_3)\cdot e_1)(n_\vartheta^{(S)}(x_3)\cdot e_3)\,dx_2\,dx_3=0,
\end{align*}
where the last equality follows from the Euler-Lagrange equation for  $\varphi_S$ tested with $\phi(\bar x):= \int_0^{x_3} (n_\vartheta^{(S)}(s)\cdot e_1)(n_\vartheta^{(S)}(s)\cdot e_3)\,ds$.

\underline{'Twist':} We show that
\begin{equation}\label{eq:b4twist}
 b_4(n_\vartheta^{(T)})=-\bar r c_S\langle\mu\rangle_\ho \cos(2\vartheta)\qquad\mbox{with}\quad c_S:=\frac{32}{\pi^3}(2\tanh(\tfrac{\pi}{2})-\pi)<0.
\end{equation}
For this, we use an explicit expression of the torsion function $\varphi_S$ in case of the square, see Section~\ref{sec:torsion}. We claim that,
\begin{align}\label{est:twistlc}
 \int_{-1}^1 ( \varphi_S(1,x_3)-\varphi_S(-1,x_3)) - 2x_3)\sin(2\vartheta+\tfrac{\pi}{2}x)\,dx=c_S \cos(2\vartheta),
\end{align}
where $c_S$ is given as in \eqref{eq:b4twist}. Before proving the identity \eqref{est:twistlc}, we observe that \eqref{est:twistlc} implies \eqref{eq:b4twist}. Indeed, $B(n_\vartheta^{(T)})_{12}=B(n_\vartheta^{(T)})_{21}=-\bar r \cos(\vartheta+\frac{\pi}4x_3)\sin(\vartheta+\frac{\pi}4x_3)=-\frac{\bar r}2\sin(2\vartheta+\frac{\pi}2x_3)$, $B(n_\vartheta^{(T)})_{13}=B(n_\vartheta^{(T)})_{31}=0$ and Lemma~\ref{L:isotropic} yield
\begin{align*}
 b_4(n_\vartheta^{(T)})=&-\bar r \langle \mu\rangle_\ho \int_S (\partial_2 \varphi_S(\bar x) - x_3)\sin(\tfrac{\pi}{2}x_3+2\vartheta)\,d\bar x\\
 =&-\bar r \langle \mu\rangle_\ho \int_{-1}^1 ( \varphi_S(1,x_3)-\varphi_S(-1,x_3)) - 2x_3)\sin(2\vartheta+\tfrac{\pi}{2}x_3)\,dx_3\\
 =&-\bar r c_S\langle\mu\rangle_\ho \cos(2\vartheta).
\end{align*}
Next, we provide the argument for \eqref{est:twistlc}. We observe that, by Lemma~\ref{L:torsion} below,
\begin{align}\label{eq:esttwistlc1}
 &\varphi_S(1,x)-\varphi_S(-1,x)\notag\\
 =&\sum_{n=1}^\infty A_n (\cosh (n\pi)-1)\left(\cos(\tfrac{n\pi}{2}(1+x))-\cos(\tfrac{n\pi}{2}(1-x))\right)\notag\\
 &+\sum_{n=1}^\infty A_n\left(\cosh(\tfrac{n\pi}{2}(1+x))-\cosh(\tfrac{n\pi}{2}(1-x))\right)\left(1-\cos(n\pi)\right),
\end{align}
%
%\begin{align}\label{eq:esttwistlc1}
% &\varphi_S(1,x)-\varphi_S(-1,x)\notag\\
% =&2\sum_{n=1}^\infty A_n (\cosh (n\pi)-1)\cos(\tfrac{n\pi}{2}(1+x))+4\sum_{n=1}^\infty A_n\cosh(\frac{n\pi}{2}(1+x)),
%\end{align}
%
where $A_n=\frac{-16}{n^3\pi^3\sinh(n\pi)}\begin{cases}1&\mbox{if $n$ is odd}\\0&\mbox{if $n$ is even}\end{cases}$. The identity $\cos(\tfrac{n\pi}{2}(1-x))=\cos(\tfrac{n\pi}{2}(1+x)-n\pi)$, \eqref{eq:esttwistlc1} and $A_n=0$ for even values of $n$ imply
\begin{align}\label{eq:esttwistlc1a}
 \varphi_S(1,x)-\varphi_S(-1,x) =&2\sum_{n=1}^\infty A_n (\cosh (n\pi)-1)\cos(\tfrac{n\pi}{2}(1+x))\notag\\
 &+2\sum_{n=1}^\infty A_n\left(\cosh(\tfrac{n\pi}{2}(1+x))-\cosh(\tfrac{n\pi}{2}(1-x))\right).
\end{align}
Straightforward calculations yield for every $n\in\N$%\footnote{M: (INTERNAL REMARK) to derive the second identity it is convenient to evaluate $\int_{-1}^1e^{ax} \sin(\tfrac{\pi}2x+2\vartheta)$ (with suitable $a$)..}
\begin{align*}
 &\int_{-1}^1\cos(\tfrac{2n-1}2\pi(1+x))\sin(\tfrac{\pi}{2}x+2\vartheta)\,dx=\begin{cases}
                                                                                                                   -\cos(2\vartheta)&\mbox{if $n=1$,}\\0&\mbox{if $n\geq2$}
                                                                                                                  \end{cases}
\\
 &\int_{-1}^1\left(\cosh(\tfrac{2n-1}2\pi(1+x))-\cosh(\tfrac{2n-1}2\pi(1-x))\right)\sin(\tfrac{\pi}{2}x+2\vartheta)\,dx\\
 =&\frac{2(2n-1)}{1+(2n-1)^2}\frac{2}{\pi}\sinh((2n-1)\pi)\cosh(2\vartheta).
 \end{align*}
Next, we plug the above two identities into \eqref{eq:esttwistlc1a} and obtain
\begin{align*}
 &\int_{-1}^1 (\varphi_S (1,x)-\varphi_S(-1,x))\sin(\tfrac{\pi}{2}x+2\vartheta)\,dx\\
 =&-2 A_1(\cosh(\pi)-1)\cos(2\vartheta)+2\sum_{n=1}^\infty A_{2n-1}\frac{2(2n-1)}{1+(2n-1)^2}\frac{2}{\pi}\sinh((2n-1)\pi)\cos(2\vartheta)\\
 =&32 \cos(2\vartheta) \left(\frac{\cosh(\pi)-1}{\pi^3 \sinh(\pi)}-\frac{4}{\pi^4}\sum_{n=1}^\infty\frac{1}{(2n-1)^2(1+(2n-1)^2)}\right)\\
 =&32 \cos(2\vartheta) \left(\frac{\cosh(\pi)-1}{\pi^3 \sinh(\pi)}-\frac{1}{2\pi^3}(\pi-2\tanh(\tfrac{\pi}{2})\right)\\
 =&\frac{32}{\pi^3} \cos(2\vartheta)\left(2\tanh(\tfrac{\pi}2)-\tfrac{\pi}2\right).
\end{align*}
Combining the last equality with $\int_{-1}^1x\sin(\tfrac{\pi}{2}x+2\vartheta)\,dx=\frac{8}{\pi^2}\cos(2\vartheta)$, we obtain \eqref{est:twistlc}.
\end{itemize}

\subsection{The torsion function for the square}\label{sec:torsion}

In this section, in contrast to all other parts of the paper, we use the notation $x=(x_1,x_2)\in\R^2$.
\begin{lemma}\label{L:torsion}
Set $Q=(-1,1)^2$. The unique function $\varphi_Q$ satisfying 
\begin{equation}\label{def:qsb}
  \min_{\varphi\in H^1(Q)}\int_Q \big((\partial_1 \varphi - x_2)^2+(\partial_2 \varphi + x_1 )^2\big)=\int_Q \big((\partial_1 \varphi_Q - x_2)^2+(\partial_2 \varphi_Q + x_1 )^2\big),
\end{equation}
and $\fint_Q \varphi_Q =0$ can be written as
\begin{align}\label{def:varphicube}
 \varphi_Q(x)=\sum_{k=0}^3 \phi_Q (R^k x)
\end{align}
where $R:=e_1\otimes e_2 - e_2 \otimes e_1$ and 
\begin{equation}\label{varphisquare}
 \phi_Q(x)=\sum_{n=1}^\infty A_{n} \cosh(\tfrac{n\pi}2(1+x_1))\cos(\tfrac{n\pi}2(1+x_2))%\quad \mbox{with }A_n:=\frac{-16}{n^3\pi^3\sinh(n\pi)}\begin{cases}1&\mbox{if $n$ is odd}\\0&\mbox{if $n$ is even}\end{cases}
\end{equation}
with 
\begin{equation}\label{def:An}
A_n:=\frac{-16}{n^3\pi^3\sinh(n\pi)}\begin{cases}1&\mbox{if $n$ is odd}\\0&\mbox{if $n$ is even}\end{cases}
\end{equation}
and the series in \eqref{varphisquare} should be interpreted as a strong $H^1(Q)$ limit of the partial sums.

\end{lemma}

\begin{proof}
\step 1 For $N\in\N$, we define $\phi^N\in C^\infty(\R^2)$ as
$$\phi^N(x):=\sum_{n=1}^N A_{2n-1} \cosh(\tfrac{(2n-1)\pi}2(x_1+1))\cos(\tfrac{(2n-1)\pi}2(x_2+1)),$$
where $A_n$ is given as in \eqref{def:An}. We claim that $\phi^N\to \phi$ strongly in $H^1(Q)$, where $\phi$ is the unique weak solution of the Neumann boundary value problem
\begin{equation}\label{eq:psi}
 \Delta\varphi=0\quad\mbox{in $Q$ and}\quad\nabla \varphi\cdot \nu=\begin{cases}
                                                                  x_2&\mbox{on $\{1\}\times (-1,1)$}\\
                                                                  0&\mbox{on $\partial Q\setminus (\{1\}\times (-1,1))$}
                                                                 \end{cases}
\end{equation}
with zero mean, i.e.\ $\int_Q \phi\,dx=0$ and where $\nu$ denotes the outer normal to $\partial Q$.

Indeed, by construction we have $\int_Q \phi^N=0$ for every $N\in\N$. To show that $(\phi^N)$ is a Cauchy sequence in $H^1(Q)$, we observe that for $N_1,N_2\in\N$ with $N_2>N_1$
\begin{align*}
\|\nabla \phi^{N_1} - \nabla \phi^{N_2}\|_{L^{\infty}(S)}\leq& \frac{16}{\pi^3}\sum_{n=N_1}^{N_2} \frac{\cosh((2n-1)\pi)}{(2n-1)^3\sinh((2n-1)\pi)}\\
\leq&\frac{16}{\pi^3\tanh(\pi)}\sum_{n=N_1}^\infty\frac1{(2n-1)^3}\to0 \qquad\mbox{as $N_1\to\infty$.}
\end{align*}
Appealing to $\int_Q \phi^N=0$ and the Poincar\'e inequality, we obtain that $(\phi^N)$ is a Cauchy sequence in $H^1(Q)$ and thus there exists $\phi\in H^1(Q)$ such that $\lim_{N\to\infty} \|\phi^N-\phi\|_{H^1(Q)}=0$. Moreover, by construction $\phi_N$ is harmonic on $Q$ for each $N\in\N$ and it holds $\nabla \phi^N \cdot \nu=0$ on $(\partial Q)\setminus (\{1\}\times (-1,1))$ for every $N\in\N$ and for all $x\in\{1\}\times (-1,1)$ 
\begin{align*}
 \nabla\phi^N(x)\cdot \nu=\partial_1 \phi^N(1,x_2)=-\frac{8}{\pi^2}\sum_{n=1}^N \frac{\cos(\frac{2n-1}{2}\pi (x_2+1))}{(2n-1)^2}=:f^N(x_2) 
\end{align*}
Altogether, we have for all $v\in C^\infty(Q)$
\begin{align*}
 \int_{-1}^1 f^N(x_2) v (1,x_2)\,dx_2 =\int_{\partial Q}\nabla \phi^N \cdot \nu v=\int_{Q}\nabla \phi^N\cdot \nabla v  
\end{align*}
Hence, the using $\nabla \phi^N\to \nabla \phi$ in $L^2(Q)$ and $f^N\to f$ with $f(x)=x$ in $L^2(-1,1)$, we obtain
\begin{equation*}
 \int_{-1}^1 x_2 v(2,x_2)\,dx_2=\int_Q \nabla \phi(x) \cdot \nabla v(x)\,dx
\end{equation*}
which is precisely the weak formulation of \eqref{eq:psi}.

\step 2 Conclusion.

In view of Step~1, $\varphi_Q$ given as in \eqref{def:varphicube} satisfy $\fint_Q\varphi_Q=0$ and (in the weak sense)
\begin{equation}\label{eq:phicube}
 -\Delta \varphi_Q=0 \quad\mbox{in $Q$}\qquad \mbox{and}\qquad \nabla \varphi_Q\cdot \nu= \begin{pmatrix}x_2\\-x_1\end{pmatrix} \cdot \nu\quad \mbox{on $\partial Q$,}
\end{equation}
where $\nu$ denotes the outer normal to $\partial Q$. Since \eqref{eq:phicube} is the Euler-Lagrange equation for \eqref{def:qsb} the claim follows.

\end{proof}

\end{document}